\documentclass[11pt, a4paper]{amsart}

\usepackage{amsthm,amsfonts,amsbsy,amssymb,amsmath,amscd}
\usepackage[all]{xy}
\usepackage[colorlinks=true]{hyperref}
\usepackage{pgf, tikz}

\newcommand{\QQ}{\mathbb{Q}}

\newcommand{\CC}{\mathbb{C}}
\newcommand{\PP}{\mathbb{P}}

\newcommand{\CE}{\mathcal{E}}
\newcommand{\CF}{\mathcal{F}}

\newcommand{\CI}{\mathcal{I}}

\newcommand{\CL}{\mathcal{L}}
\newcommand{\CO}{\mathcal{O}}
\newcommand{\CP}{\mathcal{P}}
\newcommand{\CQ}{\mathcal{Q}}
\newcommand{\CV}{\mathcal{V}}

\newcommand{\alb}{{\rm alb}}

\newcommand{\Hom}{{\mathrm{Hom}}}
\newcommand{\Pic}{{\mathrm{Pic}}}
\newcommand{\baselocus}{{\mathrm{Bs}}}
\newcommand{\movable}{{\mathrm{Mov}}}

\newcommand{\supp}{{\mathrm{Supp}}}
\newcommand{\Sym}{{\mathrm{Sym}}}

\newcommand{\ev}{{\mathrm{ev}}}

\newcommand{\vol}{{\mathrm{vol}}}
\newcommand{\rank}{{\mathrm{rank}\,}}

\newcommand{\roundup}[1]{{\left\lceil #1 \right\rceil}}
\newcommand{\rounddown}[1]{{\left\lfloor #1 \right\rfloor}}

\begin{document}

\title{Noether inequality for irregular threefolds of general type}

\author{Yong Hu}
\author{Tong Zhang}
\date{\today}

\address[Y.H.]{School of Mathematical Sciences, Shanghai Jiao Tong University, 800 Dongchuan Road, Shanghai 200240, People's Republic of China}
\email{yonghu@sjtu.edu.cn}

\address[T.Z.]{School of Mathematical Sciences, Ministry of Education Key Laboratory of Mathematics and Engineering Applications \& Shanghai Key Laboratory of PMMP,  East China Normal University, Shanghai 200241, China}
\email{tzhang@math.ecnu.edu.cn, mathtzhang@gmail.com}

\begin{abstract}
	Let $X$ be a smooth irregular $3$-fold of general type over $\CC$. We prove that the optimal Noether inequality 
	$$
	\vol(X) \ge \frac{4}{3}p_g(X)
	$$ 
	holds if $p_g(X) \ge 16$ or if $X$ has a Gorenstein minimal model. Moreover, when $X$ attains the equality and $p_g(X) \ge 16$, its canonical model can be explicitly described.
\end{abstract}

\maketitle
	
\tableofcontents
	

\theoremstyle{plain}
\newtheorem{theorem}{Theorem}[section]
\newtheorem{lemma}[theorem]{Lemma}
\newtheorem{coro}[theorem]{Corollary}
\newtheorem{prop}[theorem]{Proposition}
\newtheorem{defi}[theorem]{Definition}
\newtheorem{ques}[theorem]{Question}
\newtheorem{conj}[theorem]{Conjecture}

\newtheorem*{ques*}{Question}
\newtheorem*{conj*}{Conjecture}

\theoremstyle{remark}
\newtheorem{remark}[theorem]{Remark}
\newtheorem{assumption}[theorem]{Assumption}
\newtheorem{example}[theorem]{Example}

\numberwithin{equation}{section}
	

\section{Introduction}
Throughout this paper, all varieties are projective over the complex numbers $\CC$.

\subsection{Background} Classifying algebraic varieties is a central problem in algebraic geometry. One way to attack this problem, especially for varieties of general type, is to study the relation among their birational invariants first, and then to obtain an explicit description using these numerical information. This is often refered as the geography of algebraic varieties in the literature.

It is via this approach that many classification results for varieties of general type have been proved. As a typical example, it was proved by Noether \cite{Noether} that every smooth surface $S$ of general type satisfies the following optimal inequality:
\begin{equation} \label{eq: Noether dim 2}
	\vol(S) \ge 2p_g(S) - 4,
\end{equation}
which is now called the Noether inequality. In his celebrated paper \cite{HorikawaI}, Horikawa completely classified surfaces of general type which attain the equality in \eqref{eq: Noether dim 2}. The Noether inequality problem in dimension three also has attracted lots of attentions (e.g. \cite{Kobayashi,Chen,Catanese_Chen_Zhang,Chen_Chen,Chen_Hu}). Recently, J. Chen, M. Chen and C. Jiang \cite{Chen_Chen_Jiang} proved that every smooth $3$-fold $X$ of general type, either having a Gorenstein minimal model or with $p_g(X) \ge 11$, satisfies the following optimal Noether inequality:
\begin{equation} \label{eq: Noether dim 3}
	\vol(X) \ge \frac{4}{3} p_g(X) - \frac{10}{3}.
\end{equation}
Parallel to Horikawa's work \cite{HorikawaI}, they asked [Question 1.5, loc. cit] whether $3$-folds attaining the equality can be classified, and this question has been solved by the authors in \cite{Hu_Zhang3}. 

Motivated by the explicit classification of irregular varieties of general type, it is a fundamental question to what extent does the irregularity affect the distribution of birational invariants (see \cite[Introduction]{Persson} for a detailed description). In fact, if a smooth surface $S$ of general type satisfies the Noether equality in \eqref{eq: Noether dim 2}, then $q(S) = 0$ \cite[Theorem 10]{Bombieri}. Thus there must be a sharper Noether inequality for irregular surfaces of general type. This problem was studied extensively by Bombieri \cite[\S 10]{Bombieri} and later solved by Debarre \cite{Debarre} who proved that every irregular surface $S$ of general type satisfies the following optimal Noether inequality:
\begin{equation} \label{eq: Irregular Noether dim 2}
	\vol(S) \ge 2p_g(S).
\end{equation}
Moreover, Debarre proved that if the equality holds, then $1 \le q(S) \le 4$. 
Based on the results of Horikawa, Debarre, Catanese, Ciliberto, Mendes Lopes and Pardini \cite{HorikawaV,Debarre,Catanese_Ciliberto_Lopes,Ciliberto_Lopes_Pardini}, a complete classification of irregular surfaces of general type satisfying the equality in \eqref{eq: Irregular Noether dim 2} has been established. We refer to \cite[\S 1]{Ciliberto_Lopes_Pardini} for the history and details regarding this result.

Very similar to the surface case, it is proved recently that if a smooth $3$-fold $X$ of general type with $p_g(X) \ge 11$ satisfies the Noether equality in \eqref{eq: Noether dim 3}, then $q(X) = 0$ \cite{Hu_Zhang3}. Thus it is natural to ask:

\begin{ques*} 
	What is the optimal Noether inequality for irregular $3$-folds of general type? Once the desired inequality is obtained, can one classify the $3$-folds attaining the equality? 
\end{ques*}

However, unlike the Noether inequality \eqref{eq: Noether dim 3} which was known already two decades ago for smooth canonically polarized $3$-folds \cite{Chen}, this question is unknown even for irregular $3$-folds with smooth canonical models.

\subsection{Main results} The main purpose of this paper is to give an answer to the above question  for ``almost all" irregular $3$-folds of general type. Our first main result is the following.

\begin{theorem} \label{thm: main1}
	Let $X$ be a smooth irregular $3$-fold of general type. Then we have the following optimal Noether inequality:
	\begin{equation} \label{eq: Irregular Noether dim 3}
		\vol(X) \ge \frac{4}{3} p_g(X),
	\end{equation}
	provided one of the following conditions holds:
	\begin{itemize}
		\item [(1)] $X$ has a Gorenstein minimal model;
		\item [(2)] $q(X) = 1$ and $p_g(X) \ge 16$;
		\item [(3)] $q(X) \ge 2$.
	\end{itemize}
    Moreover, if $X$ satisfies the equality in \eqref{eq: Irregular Noether dim 3} as well as any of the above three conditions (1)--(3), then $1 \le q(X) \le 2$.
\end{theorem}

Note that minimal models of $3$-folds $X$ with $\vol(X) < \frac{4}{3}p_g(X) \le 20$ form a bounded family \cite[Corollary 2]{Martinelli_Schreieder_Tasin}. Thus Theorem \ref{thm: main1} shows that the Noether inequality \eqref{eq: Irregular Noether dim 3} holds for all minimal irregular $3$-folds of general type except possibly finitely many families.

We would like to remark on the optimality of the Noether inequality \eqref{eq: Irregular Noether dim 3}. Indeed, the authors in \cite{Hu_Zhang2} described the canonical model of all smooth irregular $3$-folds $X$ of general type with $\vol(X) = \frac{4}{3}\chi(\omega_X)$ and showed that $q(X) = 1$ and $h^2(X, \CO_X) = 0$ for such $X$. Thus we deduce that 
$$
\vol(X) = \frac{4}{3} \left(p_g(X) - h^2(X, \CO_X) + q(X) - 1\right) = \frac{4}{3}p_g(X)
$$ 
for such $X$. A notable feature is that the general Albanese fibre of such $X$ is a $(1, 2)$-surface. Here and throughout the paper, a $(1, 2)$-surface always denotes a smooth surface $S$ with $\vol(S) = 1$ and $p_g(S) = 2$. 

Our second theorem shows that ``almost all" irregular $3$-folds attaining the equality in \eqref{eq: Irregular Noether dim 3} are exactly of the above form.

\begin{theorem} \label{thm: main2}
	Let $X$ be a smooth irregular $3$-fold of general type with $\vol(X) = \frac{4}{3}p_g(X)$ and $p_g(X) \ge 16$. Then $\vol(X) = \frac{4}{3}\chi(\omega_X)$. Thus the canonical model of $X$ can be explicitly described as in \cite[Theorem 1.3]{Hu_Zhang2}.
\end{theorem}

Interestingly, a similar phenomenon actually holds also in the surface case. Namely, let $S$ be a smooth irregular surface of general type with $p_g(S) \ge 5$. Then $\vol(S) = 2p_g(S)$ if and only if $\vol(S) = 2\chi(\omega_S)$ \cite[Theorem 1.1]{Ciliberto_Lopes_Pardini}.

Results in  \cite{Kobayashi,Chen_Chen_Jiang,Hu_Zhang3} suggest that a $3$-fold $X$ of general type with small ${\vol(X)}/{p_g(X)}$
admits a pencil of $(1, 2)$-surfaces. 
Thus there should be a sharper Noether inequality for irregular $3$-folds of general type without pencils of $(1, 2)$-surfaces. Our third theorem provides a strong evidence that the desired inequality may be just like Debarre's inequality \eqref{eq: Irregular Noether dim 2}.

\begin{theorem} \label{thm: main3}
	Let $X$ be a smooth $3$-fold of general type with $q(X) \ge 2$ whose general Albanese fibre is not a $(1, 2)$-surface. Then we have the following optimal Noether inequality:
	\begin{equation} \label{eq: Irregular Noether dim 3 2}
		\vol(X) \ge 2p_g(X),
	\end{equation}
	provided one of the following conditions holds:
	\begin{itemize}
		\item [(1)] $X$ has a Gorenstein minimal model;
		\item [(2)] $p_g(X) \ne 2$ or $3$.
	\end{itemize}
\end{theorem}

The optimality of \eqref{eq: Irregular Noether dim 3 2} can be seen from the following example. Take a minimal surface $S$ of general type with $K_S^2 = 2$,  $p_g(S) = 3$ and $q(S) = 0$. Such a surface $S$ exists and has been classified by Horikawa \cite{HorikawaI}. Take a smooth curve $C$ of genus two. Let $X = S \times C$ be the product $3$-fold. Then it is easy to compute that $q(X) = 2$, $p_g(X) = 6$ and $\vol(X) = K_X^3 = 12$. In particular, we have $\vol(X) = 2p_g(X)$.

\subsection{Sketch of the proof} The whole proof of the Noether inequality is subject to the irregularity, and is divided into two main cases. Since the canonical volume and the geometric genus are birational invariants, we may assume that $X$ is a minimal irregular $3$-fold of general type. 

\textbf{Case 1}: $q(X) \ge 2$. In this case, when $\alb\dim X = 1$, we apply the Noether-Severi inequality and the slope inequality in \cite{Hu_Zhang1,Hu_Zhang2} to deduce the inequalities \eqref{eq: Irregular Noether dim 3} and \eqref{eq: Irregular Noether dim 3 2}. When $\alb\dim X \ge 2$, up to a birational modification, we choose a smooth member $S \in \movable|K_X|$ and deduce a relation of the form $\vol(X) \ge \alpha \vol(S)$ from Kawamata's extension theorem (here $\alpha$ depends on the canonical image of $X$). Since $S$ is an irregular surface of general type, we apply the Severi inequality in \cite{Pardini} or the slope inequality in \cite{Xiao} to get another relation of the form $\vol(S) \ge \beta \chi(\omega_S)$ (here $\beta$ depends on the Albanese dimension of $S$). Moreover, by a detailed analysis on the canonical image of $X$, we also prove that $p_g(X)$ is bounded from above by $\chi(\omega_S)$ and $\chi(\omega_X)$. Combining the above three inequalities with the Severi inequality in \cite{Zhang_Clifford}, the inequality \eqref{eq: Irregular Noether dim 3 2} in this case can be proved.

\textbf{Case 2}: $q(X) = 1$. In this case, the Albanese map of $X$ induces a fibration $f: X \to B$ over a genus one curve $B$. Let $F$ be a general fibre of $f$. When $F$ is not a $(1, 2)$-surface, we may slightly modify the proof in Case 1 to get the inequality \eqref{eq: Irregular Noether dim 3}. However, dealing with the case when $F$ is a $(1, 2)$-surface and $h^2(X, \CO_X) > 0$ requires new insight, mainly because the Noether-Severi inequality in \cite{Hu_Zhang2} only provides a weaker Noether inequality. Our new input here is to use the positivity of the direct image sheaf $f_* \CO_X(mK_X)$ to get more positivity of $K_X$. More precisely, we prove that $K_X - aF$ is nef for some $a > \frac{4}{3}$ when $p_g(X) \ge 16$. If $X$ is Gorenstein, we prove a weaker result that $K_X - \frac{3}{4}F$ is nef when $p_g(X) \ge 4$. These two positivity results plus a known estimate in \cite{Hu_Zhang2} help us establish the inequality \eqref{eq: Irregular Noether dim 3} in the corresponding cases.


As has been pointed out by Debarre \cite[\S 6]{Debarre}, Bombieri \cite[Lemma 14]{Bombieri} proved that $\vol(S) \ge 2\chi(\omega_S)$ holds for every irregular surface $S$ of general type. When $q(S) = 1$, we have $p_g(S) = \chi(\omega_S)$ and Debarre's inequality \eqref{eq: Irregular Noether dim 2} follows directly from Bombieri's result. As a dramatic difference, for an irregular $3$-fold $X$ with $q(X) = 1$, we only have $p_g(X) = \chi(\omega_X) + h^2(X, \CO_X)$. Thus the above implication in dimension two fails in dimension three once $h^2(X, \CO_X) > 0$, which is the main difficulty in this paper.

\subsection{Notation and conventions} \label{subsection: notation}
In this paper, we adopt the following notation and conventions.

\subsubsection*{Varieties and divisors} Let $V$ be a normal variety of dimension $d$. We say that $V$ is \emph{minimal}, if $V$ has at worst $\QQ$-factorial terminal singularities and the canonical divisor $K_V$ is nef. The \emph{geometric genus} $p_g(V)$ of $V$ is defined as
$$
p_g(V):=h^0(V, K_V).
$$
For a Weil divisor $L$ on $V$, the \emph{volume} $\vol(L)$ of $L$ is defined as
$$
\vol(L) := \limsup\limits_{n \to \infty} \frac{h^0(V, nL)}{n^d/d!}.
$$
The volume $\vol(K_V)$ is called the \emph{canonical volume} of $V$, and is denoted by $\vol(V)$. If $V$ is birational to a smooth variety $V'$ with $\vol(V') > 0$, we say that $V$ is of general type. Note that if $V$ is minimal, then $\vol(V) = K_V^d$.

For a linear system $\Lambda$ on $V$, we denote by $\movable \Lambda$ and $\baselocus \Lambda$ the movable part and the base locus of $\Lambda$, respectively. If $p_g(V) \ge 2$, the rational map
$$
\phi_{K_V}: V \dasharrow \PP^{p_g(V) - 1}
$$
induced by the linear system $|K_V|$ is called the \emph{canonical map} of $V$, and $\phi_{K_V}(V)$ is called the \emph{canonical image}.


\subsubsection*{Irregular varieties} Let $V$ be a normal variety with at worst rational singularities. We say that $V$ is \emph{irregular}, if $q(V) := h^1(V, \CO_V) > 0$. Note that $V$ has a well-defined \emph{Albanese map}
$$
a: V \to \mathrm{Alb}(V),
$$
where $A:=\mathrm{Alb}(V)$ is an abelian variety referred as the Albanese variety of $V$. The number $\dim a(V)$, denoted by $\alb\dim V$, is called the \emph{Albanese dimension} of $V$. Let 
$$
V \stackrel{f}\rightarrow W \stackrel{g}\rightarrow A
$$ 
be the Stein factorization of $a$. Then $f$ is called the \emph{Albanese fibration} of $V$, and a fibre of $f$ is called an \emph{Albanese fibre}. If $\phi: V' \to V$ is a resolution of singularities of $V$ and $a': V' \to \mathrm{Alb}(V')$ is the Albanese map of $V'$, then $\mathrm{Alb}(V') \cong \mathrm{Alb}(V)$ and $a' = a \circ \phi$.


\subsection*{Acknowledgments}
Both authors would like to thank Professors Jungkai Alfred Chen and Zhi Jiang for their interests in this paper.

Y.H. is supported by National Key Research and Development Program of China \#2023YFA1010600 and the National Natural Science Foundation of China (Grant No. 12201397). T.Z. is supported by the National Natural Science Foundation of China (Grant No. 12071139), Science and Technology Commission of Shanghai Municipality (No. 22JC1400700, No. 22DZ2229014), and Fundamental Research Funds for the Central Universities.

\section{Preliminaries}

In this section, we list some preliminary results used in the later proof.

\subsection{Linear systems on $3$-folds}
Let $V$ be a normal $3$-fold. Let $L$ be a Cartier divisor on $V$ with $h^0(V, L) \ge 2$. Let
$$
\phi_L: V \dasharrow \PP^{h^0(V, L) - 1}
$$
be the rational map induced by the linear system $|L|$. Denote $\Sigma_L = \phi_L(V)$.

We first recall a lemma due to Reid.

\begin{lemma} \cite[Corollaries (i)]{Reid_2} \label{lem: Reid90}
	Let $W \subseteq \PP^N$ be an irreducible and non-degenerate variety of dimension $w$. Then
	$$
	h^0\left(\PP^N, \CI_W(2)\right) \le \binom{N-w+2}{2} - \min \{\deg W, 2(N-w) + 1\}.
	$$
\end{lemma}

\begin{lemma} \label{lem: h^0(2L)}
	If $\dim \Sigma_L = 3$, then 
	\begin{align*}
		h^0(V, 2L) & \ge 3h^0(V, L) - 3 + \min\{\deg \Sigma_L, 2h^0(V, L) - 7\} \ge 4h^0(V, L) - 6.
	\end{align*}
	If $\dim \Sigma_L = 2$, then
	\begin{align*}
		h^0(V, 2L) \ge 2h^0(V, L) - 1 + \min\{\deg \Sigma_L, 2h^0(V, L) - 5\} \ge 3h^0(V, L) - 3.
	\end{align*}
\end{lemma}

\begin{proof}
	Up to a birational modification, we may assume that $V$ is smooth and $|L|$ is base point free. Write $N = h^0(V, L) - 1$. By the assumption, $\Sigma_L \subseteq \PP^N$ is non-degenerate.
	
	If $\dim \Sigma_L = 3$, then $h^0(V, L) \ge 4$ and $\deg \Sigma_L \ge N-2 = h^0(V, L) - 3$. By Lemma \ref{lem: Reid90}, we deduce that
	\begin{align*}
		h^0 \left(\Sigma_L, \CO_{\Sigma_L}(2)\right) & \ge h^0\left(\CO_{\PP^N}, \CO_{\PP^N}(2)\right) - h^0\left(\PP^N, \CI_{\Sigma_L}(2)\right) \\
		& \ge \binom{N+2}{2} - \binom{N-1}{2} + \min\{\deg \Sigma_L, 2N - 5\} \\
		& = 3h^0(V, L) - 3 + \min\{\deg \Sigma_L, 2h^0(V, L) - 7\} \\
		& \ge 4h^0(V, L) - 6.
	\end{align*}
    If $\dim \Sigma_L = 2$, then $h^0(V, L) \ge 3$ and $\deg \Sigma_L \ge N - 1 = h^0(V, L) - 2$. By a similar argument via Lemma \ref{lem: Reid90}, we deduce that
    $$
    h^0 \left(\Sigma_L, \CO_{\Sigma_L}(2)\right) \ge 2h^0(V, L) - 1 + \min\{\deg \Sigma_L, 2h^0(V, L) - 5\}
    \ge 3h^0(V, L) - 3.
    $$
    Note that in either case, we have $h^0(V, 2L) \ge h^0 \left(\Sigma_L, \CO_{\Sigma_L}(2)\right)$. Thus the proof is completed.
\end{proof}

\begin{lemma} \label{lem: extension}
	Suppose that $V$ is minimal of general type and $\pi: V' \to V$ be a resolution of singularities of $V$. Let $S$ be a semi-ample divisor on $V'$ which itself is a smooth surface of general type. If $\lambda\pi^*K_V - S$ is $\QQ$-linear equivalent to an effective $\QQ$-divisor for some $\lambda > 0$, then 
	$$
	\lambda(1 + \lambda)^2 K_V^3 \ge K_{S_0}^2,
	$$
	where $\sigma: S \to S_0$ is the contraction onto the minimal model $S_0$ of $S$.
\end{lemma}

\begin{proof}
	Since $K_V$ is semi-ample, by \cite[Corollary 2.3]{Chen_Chen_Jiang} which was deduced from Kawamata's extension theorem, the $\QQ$-divisor $(1 + \lambda)\pi^*K_V|_S - \sigma^*K_{S_0}$ is pseudo-effective. Thus
	$$
	\lambda(1 + \lambda)^2 K_V^3 \ge (1 + \lambda)^2 \left((\pi^*K_V)^2 \cdot S \right) \ge (\sigma^*K_{S_0})^2 = K_{S_0}^2.
	$$
	The proof is completed.
\end{proof}

\subsection{Positivity of divisors on fibred varieties} 
Let 
$$
f: V \to B
$$ 
be a fibration from an $n$-dimensional normal variety $V$ onto a smooth curve $B$. Denote by $F$ a general fibre of $f$. Let $L$ be a $\QQ$-Cartier Weil divisor on $V$ such that $L$ is $f$-nef and that $L|_F$ is Cartier. Suppose that $f_* \CO_V(L)$ is nonzero. Let 
$$
0 = \CE_0 \subsetneq \CE_1 \subsetneq \cdots \subsetneq \CE_l = f_* \CO_V(L)
$$
be the Harder-Narasimhan filtration of $f_* \CO_V(L)$. For any $1 \le i \le l$, set 
$$
r_i= \rank \CE_i, \quad  \mu_i
=\frac{\deg(\CE_i/\CE_{i-1})}{\rank (\CE_i / \CE_{i-1})}.
$$
Then we have $\mu_1 > \mu_2 > \cdots > \mu_l$.

Consider the evaluation morphism 
$$
\ev: f^*f_*\CO_V(L) \to \CO_V(L).
$$ 
It induces a rational map
$$
\psi: V\dashrightarrow \PP_B(f_*\CO_{V}(L)).
$$
Note that $\psi$ is actually defined by the linear system $|L + tF|$ for a sufficiently large $t \in \mathbb{N}$. Denote $\Delta = \mathrm{Sing}(V) \cup \mathrm{Supp}(\CO_V(L)/\mathrm{Im}(\ev))$.


\begin{lemma} \label{lem: nefness}
	Let $C$ be an integral curve on $V$ not contained in the $f$-horizontal part of $\Delta$. Then 
	$
	\left((L - \mu_l F) \cdot C\right) \ge 0.
	$
\end{lemma}

\begin{proof}
	Suppose that $C$ is $f$-vertical. Since $L$ is $f$-nef, we have
	$$
	\left((L - \mu_l F) \cdot C\right) = (L \cdot C) \ge 0.
	$$
	Suppose that $C$ is $f$-horizontal now. Choose a birational modification $\pi: V' \to V$ such that $V'$ is smooth, $\pi$ is isomorphic over $V \backslash \Delta$, and $|M| = \movable|\rounddown{\pi^*(L+tF)}|$ is base point free for a sufficiently large $t \in \mathbb{N}$.
	Write 
	$$
	\pi^*(L+tF) = M + Z.
	$$
	Then we have $\pi(Z) \subseteq \Delta$. Since $C$ is not contained in $\Delta$, the strict transform $C'$ of $C$ under $\pi$ is not contained in $\supp(Z)$. Now by \cite[Corollary 3.5]{Miyaoka}, the $\QQ$-divisor $\pi^*(L - \mu_l F) - Z$ is nef. It follows that
	$$
	\left((L - \mu_l F) \cdot C\right) = \left(\pi^*(L - \mu_l F) \cdot C'\right) \ge (Z\cdot C') \ge 0.
	$$
	Thus the proof is completed.
\end{proof}

\subsection{Vector bundles on elliptic curves}
Let $E$ be an elliptic curve. For a vector bundle $\CE$ on $E$, we may write 
$$
\CE = \bigoplus_{i=1}^n \CE_i,
$$
where each $\CE_i$ is indecomposable, i.e., not a direct sum of subbundles. By \cite[Page 3, Fact]{Tu}, every $\CE_i$ is semistable. Thus it is easy to obtain the Harder-Narasimhan filtration of $\CE$. Indeed, we may further assume that 
$$
\mu(\CE_1) = \cdots = \mu(\CE_{i_1}) > \cdots > \mu(\CE_{i_{k-1}+1}) = \cdots = \mu(\CE_{i_k}),
$$
where $\mu(\CE_i) = \deg \CE_i / \rank \CE_i$ for each $i$ and $i_k = n$. Then the Harder-Narasimhan filtration of $\CE$ is just
$$
0 = \CF_0 \subsetneq \CF_1 \subsetneq \cdots \subsetneq \CF_k = \CE,
$$
where $\CF_l = \CE_1 \oplus \cdots \oplus \CE_{i_l}$ for any $1 \le l \le k$. 

\begin{defi}
	For a vector bundle $\CE$ on $E$ as above, the minimal (resp. maximal) slope $\mu_{\min}(\CE)$ (resp. $\mu_{\max}(\CE)$) of $\CE$ is defined to be $\min \{\mu(\CE_i)| 1 \le i \le n\}$ (resp. $\max \{\mu(\CE_i)| 1 \le i \le n\}$).
\end{defi}


\begin{lemma} \label{lem: vector bundle tensor}
	Let $\CV_1$ and $\CV_2$ be two vector bundles on $E$. Then
	$$
	\mu_{\min}(\CV_1 \otimes \CV_2) = \mu_{\min}(\CV_1) + \mu_{\min}(\CV_2).
	$$
\end{lemma}

\begin{proof}
	See \cite[Page 29]{Huybrechts_Lehn}.
\end{proof}

\begin{lemma}  \label{lem: vector bundle hom}
	Let $\CV_1$ and $\CV_2$ be two vector bundles on $E$ such that $\CV_2$ is indecomposable. If $\Hom(\CV_1, \CV_2) \neq 0$, then 
	$
	\mu(\CV_2) \ge \mu_{\min}(\CV_1).
	$
\end{lemma}

\begin{proof}
	Since $\CV_2$ is indecomposable, it is semistable, and $\mu_{\max}(\CV_2) = \mu(\CV_2)$. Thus this lemma follows from \cite[Lemma 1.3.3]{Huybrechts_Lehn}.
\end{proof}

\begin{lemma} \label{lem: vector bundle surjection}
	If there is a generically surjective morphism $\phi: \CV_1 \to \CV_2$ between vector bundles $\CV_1$ and $\CV_2$ on $E$, then 
	$
	\mu_{\min}(\CV_2) \ge \mu_{\min}(\CV_1).
	$
\end{lemma}

\begin{proof}
	Let $\CF$ be a direct summand of $\CV_2$ with $\mu(\CF) = \mu_{\min}(\CV_2)$. Since $\phi$ is generically surjective, the composition $\CV_1 \to \CF$ of $\phi$ and the natural projection $\CV_2 \to \CF$ is nonzero. Then this lemma follows from \cite[Lemma 1.3.3]{Huybrechts_Lehn}.
\end{proof}


\begin{lemma} \cite[Lemma 4.3]{Chen_Hacon2} \label{lem: vector bundle IT0}
	Let $\CF$ be a vector bundle on $E$. Then $\CF$ is ample if and only if $H^1(E, \CF \otimes P)=0$ for all $P \in \Pic^0(E)$.
\end{lemma}

\begin{coro} \label{coro: vector bundle ample}
	Let $f: V \to E$ be a fibration from a smooth variety $V$ to $E$, and let $L$ be a nef and big $\QQ$-divisor with its fractional part being simple normal crossing. Then the sheaf $f_* \CO_V(K_V + \roundup{L})$ is ample.
\end{coro}

\begin{proof}
	By the Kawamata-Viehweg vanishing theorem, for any $P \in \Pic^0(E)$, we have $h^1 \left(V, \CO_V \left(K_V + \roundup{L} \right) \otimes f^*P \right) = 0$.
	By the Leray spectral sequence, we deduce that $h^1 (E, f_* \CO_V(K_V + \roundup{L}) \otimes P) = 0$. By Lemma \ref{lem: vector bundle IT0}, $f_* \CO_V(K_V + \roundup{L})$ is ample.
\end{proof}

\section{Irregular $3$-folds of general type with $q \ge 2$}

Throughout this section, we assume that $X$ is a minimal $3$-fold of general type with $q(X) \ge 2$. The main result in this section is the following theorem.

\begin{theorem} \label{thm: q>1}
	Let $X$ be as above. Then
	$$
	K_X^3 \ge \frac{4}{3}p_g(X),
	$$ 
	and the equality holds only when $\alb \dim X = 2$ and $q(X) = 2$. 
	Suppose furthermore that the general Albanese fibre of $X$ is not a $(1,2)$-surface. Then
	$$
	K_X^3 \ge 2p_g(X)
	$$
	if $X$ is Gorenstein or $p_g(X) \ne 2, 3$, 
\end{theorem}

The rest of this section is devoted to the proof.


\subsection{The case when $\alb \dim X = 1$}
Suppose first that $\alb \dim X = 1$. Let
$$
f: X \to B
$$
be the Albanese fibration of $X$, where $B$ is a smooth curve of genus $g(B) = q(X) \ge 2$. Let $F$ be a general fibre of $f$.

\begin{prop} \label{prop: albdim 1 q>1}
	Suppose that $\alb \dim X = 1$. Then $K^3_X > \frac{4}{3} p_g(X)$ if $F$ is a $(1, 2)$-surface, and $K_X^3 \ge 2p_g(X)$ otherwise. In particular, Theorem \ref{thm: q>1} holds.
\end{prop}

\begin{proof}
	We first assume that $F$ is a $(1, 2)$-surface. If $h^2(X, \CO_X) \le 1$, by \cite[Theorem 1.1]{Hu_Zhang2}, we have
	$$
	K_X^3 > \frac{4}{3} \chi(\omega_X) = \frac{4}{3} p_g(X) + \frac{4}{3}\left(q(X) - h^2(X, \CO_X) - 1\right) \ge \frac{4}{3} p_g(X).
	$$
	Now suppose that $h^2(X, \CO_X) \ge 2$. Since $q(F) = 0$ and $R^1 f_*\omega_X$ is torsion free by \cite[Theorem 2.1]{Kollar_Higher_direct_image}, we know that $R^1f_* \omega_X = 0$. By the Leray spectral sequence, we have $h^2(X, \CO_X) = h^1(X, \omega_X) = h^1(B, f_*\omega_X)$. Since $f_* \omega_{X/B}$ is of rank two and $h^1(B, f_*\omega_X) \ge 2$, Fujita's decomposition \cite[Theorem 3.1]{Fujita} forces $f_* \omega_{X/B} = \CO_B^{\oplus 2}$. Thus $f_*\omega_X = \omega_B^{\oplus 2}$. In particular, $p_g(X) = 2g(B) \ge 4$. Note that $K_{X/B}$ is nef. It follows that
	$$
	K_X^3 = K_{X/B}^3 + 6\left(g(B) - 1\right) \ge 6\left(g(B) - 1\right) = 3p_g(X) - 6 > \frac{4}{3}p_g(X).
	$$


	We now assume that $F$ is not a $(1, 2)$-surface. By \cite[Theorem 1.6]{Hu_Zhang1}, $K_{X/B}^3 \ge 2 \deg f_* \omega_{X/B}$. Thus it follows that
	\begin{align*}
		K^3_X & = K_{X/B}^3 + 6\left(g(B) - 1\right)K_F^2 \\
		& \ge 2 \deg f_* \omega_{X/B} + 6\left(g(B) - 1\right)K_F^2 \\
		& = 2 \deg f_* \omega_X + \left(6K_F^2 - 4p_g(F)\right) \left(g(B) - 1\right) \\ 
		& = 2p_g(X) - 2h^1(B, f_*\omega_X) + \left(6K_F^2 - 2p_g(F)\right) \left(g(B) - 1\right),
	\end{align*}
	where the last equality follows from the Riemann-Roch theorem. By \cite[Theorem 3.1]{Fujita}, $h^1(B, f_*\omega_X) \le \rank f_*\omega_X = p_g(F)$. Moreover, since $K_F^2 \ge 2p_g(F) - 4$, it is easy to see that $3K_F^3 \ge 2p_g(F)$ unless $F$ is a $(1, 2)$-surface. Combine the above inequalities together, and we deduce that
	$$
	K^3_X \ge 2p_g(X) - 2p_g(F) + 6K_F^2 - 2p_g(F) \ge 2p_g(X).
	$$
	Thus the proof is completed.
\end{proof}

\subsection{The case when $\alb \dim X \ge 2$} \label{subsection: albdim 2} We start from the following result.

\begin{prop} \label{prop: albdim 2 pg=1}
	Suppose that $\alb \dim X \ge 2$ and $p_g(X) = 1$. Then $K_X^3 \ge 2$.
\end{prop}

\begin{proof}
	By \cite[Theorem 1.4 and Proposition 4.3]{Jiang}, we have $K_X^3 \ge 12$ when $\alb \dim X = 3$ and $K_X^3 \ge 2$ when $\alb \dim X = 2$. Thus the result follows.
\end{proof}

In the rest of this subsection, we assume that $\alb \dim X \ge 2$ and $p_g(X) \ge 2$. By \cite[Corollary 1.9]{Zhang_Clifford}, we have the following Severi inequality
\begin{equation} \label{eq: Severi q>1}
	K_X^3 \ge 4\chi(\omega_X)
\end{equation}
which will be useful in the proof. Let $\pi: X' \to X$ be a birational modification such that $X'$ is smooth and $|M| = \movable|\rounddown{\pi^*K_X}|$ is base point free. 
Then we have the following commutative diagram:
$$
\xymatrix{
	X' \ar[d]_{\pi}  \ar[rr]^{\psi} \ar[drr]^{\phi_{M}} & &  \Sigma' \ar[d]^{\tau}  \\
	X  \ar@{-->}[rr]_{\phi_{K_X}} & & \Sigma         
}
$$
where $\Sigma \subset \PP^{p_g(X) - 1}$ is the canonical image of $X$, $\phi_M: X' \to \Sigma$ is the morphism induced by $|M|$, and $X' \stackrel {\psi}\rightarrow \Sigma' \stackrel{\tau} \rightarrow \Sigma$ is the Stein factorization of $\phi_M$. Moreover, if $\dim \Sigma \ge 2$, by Bertini's theorem, we may take a smooth general member $S \in |M|$. Let $\sigma: S \to S_0$ be the contraction onto the minimal model of $S$. By Lemma \ref{lem: extension} for $\lambda = 1$, we have
\begin{equation} \label{eq: volume extension}
	4K_X^3 \ge K_{S_0}^2. \footnote{This inequality holds once the canonical image has dimension at least two and will also be used in Section \ref{section: q=1 II}.}
\end{equation}

The notation here will be used throughout this subsection.

\subsubsection{$\dim \Sigma = 3$} We first treat the case when $\dim \Sigma = 3$. Thus $p_g(X) \ge 4$. 

\begin{prop} \label{prop: albdim 2 canonical dim 3}
	Suppose that $\alb \dim X \ge 2$ and $\dim \Sigma = 3$. Then $K_X^3 \ge 2p_g(X)$.
\end{prop}

\begin{proof} 
	Since $\phi_M$ is generically finite, we know that $S$ is a nef and big divisor. By the Kawamata-Viehweg vanishing theorem, $h^i(X', K_{X'} + S) = 0$ for all $i > 0$. Together with the short exact sequence
	\begin{equation} \label{eq: ses}
		0 \to \CO_{X'}(K_{X'}) \to \CO_{X'}(K_{X'} + S) \to \CO_S(K_S) \to 0,
	\end{equation}
	it follows that
	$$
	\chi(\omega_X) + \chi(\omega_S) = \chi\left(\CO_{X'}(K_{X'} + S)\right) = h^0(X', K_{X'} + S) \ge h^0(X', 2S).
	$$
	By Lemma \ref{lem: h^0(2L)}, $h^0(X', 2S) \ge 4h^0(X', S) - 6 = 4p_g(X) - 6$. Thus
	\begin{equation} \label{eq: albdim 2 canonical dim 3 chi}
		\chi(\omega_X) + \chi(\omega_S) \ge 4p_g(X) - 6.
	\end{equation}
	
	Since $S$ is a big divisor, we know that $\alb \dim S = 2$ and the Severi inequality $K_{S_0}^2 \ge 4\chi(\omega_S)$ holds \cite[Theorem 2.1]{Pardini}. Together with \eqref{eq: volume extension}, we have
	\begin{equation} \label{eq: Severi extension}
		K_X^3 \ge \chi(\omega_S).
	\end{equation}
	Combine the above inequalities with \eqref{eq: Severi q>1}. It follows that
	$$
	K_X^3 \ge \frac{4}{5} \left(\chi(\omega_X) + \chi(\omega_S) \right) \ge \frac{16}{5}p_g(X) - \frac{24}{5} \ge 2p_g(X),
	$$
	where the last inequality holds because $p_g(X) \ge 4$. Thus the proof is completed.
\end{proof}

\subsubsection{$\dim \Sigma = 2$} \label{subsubsection: albdim 2 canonical dim 2} Here we treat the case when $\dim \Sigma = 2$. Thus $p_g(X) \ge 3$. Denote by $C$ a general fibre of $\psi$. Then $C$ is a smooth curve of genus $g(C) \ge 2$. Up to a birational modification, we may assume that both $X'$ and $\Sigma'$ are smooth. Now $S = \phi_M^*H$ for a general hyperplane section $H$ on $\Sigma$, and it is no longer a big divisor. Write $H' = \tau^*H$. We may assume $H'$ is smooth. Let $d = (\deg \tau) \cdot (\deg \Sigma)$. Then $M|_S \equiv dC$. Since $\Sigma$ is non-degenerate, we have 
$$
\deg \Sigma \ge p_g(X) - 2.
$$

\begin{lemma} \label{lem: albdim 2 vanishing}
	We have $h^i(\Sigma', R^j \psi_* \CO_{X'}(K_{X'} + S)) = 0$ for all $i > 0$ and $j \ge 0$. Moreover, $h^2(X', K_{X'} + S) = 0$.
\end{lemma}

\begin{proof}
	Since $H'$ is nef and big, the first vanishing follows from \cite[Corollary 10.15]{Kollar_Shafarevich}. Moreover, it implies that $h^2(X', K_{X'} + S) = h^0(\Sigma', R^2 \psi_* \omega_{X'} \otimes \CO_{\Sigma'}(H'))$ by the Leray spectral sequence. Since $R^2 \psi_* \omega_{X'} = 0$ \cite[Theorem 2.1]{Kollar_Higher_direct_image}, the second result follows.
\end{proof}

\begin{lemma} \label{lem: albdim 2 canonical dim 2 q}
	We have $q(S) \ge q(X)$.
\end{lemma}

\begin{proof}
	From \eqref{eq: ses}, we obtain the following exact sequence:
	$$
	H^1(S, K_S) \to H^2(X', K_{X'}) \to H^2(X', K_{X'}+S).
	$$
	Note that $q(S) = h^1(S, K_S)$ and $q(X) = h^2(X', K_{X'})$. It follows from Lemma \ref{lem: albdim 2 vanishing} that $q(S) \ge q(X)$.
\end{proof}

\begin{lemma} \label{lem: albdim 2 canonical dim 2 chi}
	We have
	$$
	\chi(\omega_X) + \chi(\omega_S) \ge 2p_g(X) + \min\{\deg \Sigma, 2p_g(X) - 5\} - h^0(\Sigma', K_{\Sigma'} + H') - 1.
	$$
\end{lemma}

\begin{proof}
	As in the proof of Proposition \ref{prop: albdim 2 canonical dim 3}, we have
	$$
	\chi(\omega_X) + \chi(\omega_S) = \chi\left(\CO_{X'}(K_{X'} + S)\right).
	$$
	By \cite[Proposition 7.6]{Kollar_Higher_direct_image}, $R^1 \psi_* \omega_{X'} = \omega_{\Sigma'}$. Applying the Leray spectral sequence and Lemma \ref{lem: albdim 2 vanishing}, we deduce that
	\begin{align*}
		\chi\left(\CO_{X'}(K_{X'} + S)\right) & = h^0(X', K_{X'} + S) - h^1(X', K_{X'} + S) \\
		& = h^0(X', K_{X'} + S) - h^0\left(\Sigma', R^1 \psi_* \omega_{X'} \otimes \CO_{\Sigma'}(H') \right) \\
		& \ge h^0(X', 2S) - h^0(\Sigma', K_{\Sigma'} + H').
	\end{align*}
	Moreover, since $p_g(X) = h^0(X', S)$, by Lemma \ref{lem: h^0(2L)}, we have
	$$
	h^0(X', 2S) \ge 2p_g(X) - 1 + \min\{\deg \Sigma, 2p_g(X) - 5\}.
	$$
	Thus the proof is completed.
\end{proof}

\begin{lemma} \label{lem: albdim 2 canonical dim 2 case 1}
	Suppose that $\alb \dim X \ge 2$ and $\dim \Sigma = 2$. Then 
	$$
	K_X^3 \ge \min \left\{2p_g(X), 4p_g(X) - 8, 3p_g(X) - \rounddown{\frac{p_g(X) - 1}{2}} - 3\right\}
	$$
	provided one of the following holds:
	\begin{itemize}
		\item [(i)] $g(H') = 0$;
		\item [(ii)] $g(H') = 1$ and $\deg \Sigma = p_g(X) - 1 \ge 3$.
	\end{itemize}
\end{lemma}

\begin{proof}
	Suppose that $K_X^3 < 2p_g(X)$. By \eqref{eq: Severi q>1}, we may assume that $2\chi(\omega_X) + 1 \le p_g(X)$. Since $H'$ is nef and big with $g(H') \le 1$, by the adjunction, we have $(K_{\Sigma'} \cdot H') < 0$. Thus $p_g(\Sigma') = 0$. By the following short exact sequence 
	$$
	0\to \CO_{\Sigma'}(K_{\Sigma'}) \to \CO_{\Sigma'}(K_{\Sigma'}+H')\to K_{H'}\to 0,
	$$
	we deduce that $h^0(\Sigma', K_{\Sigma'}+H') \le g(H')$. 
	Thus by Lemma \ref{lem: albdim 2 canonical dim 2 chi}, we deduce that under both assumptions (i) and (ii),
	$$
	\chi(\omega_S) \ge 3p_g(X) - \chi(\omega_X) - 3 \ge 3p_g(X) - \rounddown{\frac{p_g(X) - 1}{2}} - 3.
	$$
	
	Suppose that $\alb \dim S = 2$. Then \eqref{eq: Severi extension} also holds here. Thus we have
	$$
	K_X^3 \ge \chi(\omega_S) \ge 3p_g(X) - \rounddown{\frac{p_g(X) - 1}{2}} - 3.
	$$ 
	Suppose that $\alb \dim S < 2$. By Lemma \ref{lem: albdim 2 canonical dim 2 q}, $q(S) \ge 2$. Let $h: S \to B$ be the Albanese fibration of $S$ with general fibre $D$. Then $g(B) \ge 2$. In particular, $h$ descends to a fibration $h_0: S_0 \to B$ via $\sigma$. Since $g(H')\le 1$ and $g(D) \ge 2$, $D$ is neither a fibre nor a section of $\psi|_S: S \to H'$. Thus $(D \cdot C) \ge 2$. Note that $M|_S \equiv dC$. We deduce that
	$$
	\left( (\sigma^*K_{S_0}) \cdot D\right) = (K_S \cdot D) \ge \left((K_{X'} + M)|_S \cdot D \right) \ge 2(M|_S \cdot D) \ge 4d.
	$$
	Moreover, since $K_{S_0/B}$ is nef, we deduce that
	$$
	K_{S_0}^2 = K_{S_0/B}^2 + 4 \left(g(B) - 1\right) \left( (\sigma^*K_{S_0}) \cdot D\right) \ge 16d.
	$$
	Thus by \eqref{eq: volume extension}, we have
	$$
	K_X^3 \ge \frac{1}{4} K_{S_0}^2 \ge 4d \ge 4p_g(X) - 8.
	$$
	The proof is completed.
\end{proof}

\begin{lemma} \label{lem: albdim 2 canonical dim 2 case 2}
	Suppose that $\alb \dim X \ge 2$, $\dim \Sigma = 2$ and $g(H') \ge 1$. Then $K_X^3 \ge 2d$.
\end{lemma}

\begin{proof}
	Since $g(H') \ge 1$, the fibration $h=\psi|_S: S \to H'$ naturally descends to a fibration $h_0: S_0 \to H'$ via $\sigma$. We still denote by $C$ a general fibre of $h_0$.
	
	Suppose first that $\alb \dim S = 2$. Then $q(S) > g(H')$. By \cite[Corollary 1]{Xiao}, we know that
	$$
	K_{S_0}^2 - 8 \left( g(C) - 1 \right) \left( g(H') - 1 \right) = K_{S_0/H'}^2 \ge 4 \deg h_* \omega_{S/H'}.
	$$
	On the other hand, by the adjunction, $K_S - 2M|_S$ is an effective divisor on $S$. Thus $h_* \CO_S(2M|_S) \otimes \omega_{H'}^{-1}$ is a subsheaf of $h_* \omega_{S/H'}$. Let $\CL$ be the saturation of $h_* \CO_S(2M|_S) \otimes \omega_{H'}^{-1}$ in $h_* \omega_{S/H'}$. Then we obtain a short exact sequence
	$$
	0 \to \CL \to h_* \omega_{S/H'} \to \CQ \to 0,
	$$
	where $\CL$ is a line bundle with 
	$$
	\deg \CL \ge \deg \left(h_* \CO_S(2M|_S) \otimes \omega_{H'}^{-1}\right) = 2d - 2 \left(g(H') - 1\right).
	$$ 
	By \cite[Theorem 2.7]{Fujita}, $h_* \omega_{S/H'}$ is nef. Thus $\deg \CQ \ge 0$. It follows that 
	$$
	\deg h_* \omega_{S/H'} \ge \deg \CL  \ge 2d - 2 \left(g(H') - 1\right).
	$$
	Combine the above inequalities with \eqref{eq: volume extension}. It follows that
	$$
	K_X^3 \ge \frac{1}{4} K_{S_0}^2 \ge 2d + \left( 2 g(C) - 4 \right)\left(g(H') - 1\right) \ge 2d. 
	$$
	
	Now suppose that $\alb \dim S < 2$. Since $g(H') \ge 1$, by the universal property of the Albanese map, we know that $h: S \to H'$ is just the Albanese fibration of $S$. Let $a': X' \to A$ be the Albanese map of $X'$. Since $a'|_S$ factors through $h$ and $S$ is general, we deduce that $\dim a'(X') = 2$, and $C$ is a general Albanese fibre of $X'$. Moreover, $a'$ factors through $\psi$. Thus $\Sigma' \to A$ is generically finite onto its image. Let $E_\pi = K_{X'} - \pi^*K_X$ be the $\pi$-exceptional divisor. Since $E_\pi$ is covered by rational curves \cite[Corollary 1.5]{Hacon_Mckernan}, we know that $(E_\pi \cdot C) = 0$. Otherwise, $E_\pi \to \Sigma'$ would be dominant, which is absurd. Thus we have $\left((\pi^*K_X) \cdot C\right) = (K_{X'} \cdot C) = 2g(C) - 2$. It follows that
	$$
	K_X^3 \ge \left((\pi^*K_X) \cdot M^2\right) = d \left((\pi^*K_X) \cdot C\right) = d \left(2g(C) - 2\right) \ge 2d.
	$$
	Thus the whole proof is completed.
\end{proof}

\begin{prop} \label{prop: albdim 2 canonical dim 2}
	Suppose that $\alb \dim X \ge 2$ and $\dim \Sigma = 2$.  
	\begin{itemize}
		\item [(1)] If $p_g(X) \ge 4$, then $K_X^3 \ge 2p_g(X)$;
		\item [(2)] If $p_g(X) = 3$, then $K_X^3 \ge 4$.
	\end{itemize}
\end{prop}

\begin{proof}
	If $g(H') = 0$, then the result follows from Lemma \ref{lem: albdim 2 canonical dim 2 case 1}. If $g(H') \ge 1$ and $d \ge p_g(X)$, then the result follows from Lemma \ref{lem: albdim 2 canonical dim 2 case 2}. Thus we only need to treat the case when $g(H') \ge 1$ and $d \le p_g(X) - 1$. In this case, we have $p_g(X) - 2 \le \deg \Sigma \le p_g(X) - 1$.
	
	If $\deg \Sigma = p_g(X) - 1$, then $\tau$ is birational and $p_g(X) \ge 4$. By \cite[Theorem 8]{Nagata}, we know that $g(H') = 1$. Thus by Lemma \ref{lem: albdim 2 canonical dim 2 case 1}, we have
	$$
	K_X^3 \ge 2p_g(X).
	$$ 
	If $\deg \Sigma = p_g(X) - 2$, since $g(H') \ge 1$, by \cite[Theorem 7]{Nagata}, we know that $\deg \tau \ge 2$. Now $p_g(X) - 1 \ge d = (\deg \tau) \cdot (\deg \Sigma) \ge 2p_g(X) - 4$. We deduce that $p_g(X) = 3$ and $d = 2$. By Lemma \ref{lem: albdim 2 canonical dim 2 case 2}, we have
	$$
	K_X^3 \ge 2d = 4.
	$$
	The proof is completed.
\end{proof}

\subsubsection{$\dim \Sigma = 1$} Finally, we treat the case when $\dim \Sigma = 1$. 

\begin{prop} \label{prop: albdim 2 canonical dim 1 case 1}
	Suppose that $\alb \dim X \ge 2$ and $\dim \Sigma = 1$.
	\begin{itemize}
		\item [(1)] If $p_g(X) \ge 4$, then $K_X^3 \ge 2p_g(X)$;
		\item [(2)] If $p_g(X) = 3$, then $K_X^3 \ge \frac{16}{3}$.
	\end{itemize}
\end{prop}

\begin{proof}
	Since $\dim \Sigma = 1$, we have
	$
	M \equiv b S,
	$
	where $S$ is a general fibre of $\psi$. 
	
	We first consider the case when $g(\Sigma') > 0$. In this case, every fibre of $\pi$ is rationally chain connected by \cite[Corollary 1.5]{Hacon_Mckernan}. Thus $\psi$ naturally descends to a fibration $h: X \to \Sigma'$. By the adjunction, the general fibre of $h$ is smooth and minimal. Thus we may assume that $S_0$ is a general fibre of $h$. Since $\alb \dim X \ge 2$, we know that $q(S_0) > 0$. By \cite[Lemma 14]{Bombieri}, $K_{S_0}^2 \ge 2$. On the other hand, $g(\Sigma') > 0$ implies that $b \ge p_g(X)$. Thus it follows that
	$$
	K_X^3 \ge b(K_X^2 \cdot S_0) = bK_{S_0}^2 \ge 2p_g(X).
	$$
	
	From now on, we assume that $g(\Sigma') = 0$. Then every vector bundle over $\Sigma'$ splits. In this case, $M \sim (p_g(X) - 1) S$. By Lemma \ref{lem: extension} for $\lambda = \frac{1}{p_g(X) - 1}$, when $K_{S_0}^2 \ge 6$ and $p_g(X) \ge 3$, we always have 
	$$
	K_X^3 \ge \frac{\left(p_g(X) - 1\right)^3}{\left(p_g(X)\right)^2} K_{S_0}^2 \ge \frac{6\left(p_g(X) - 1\right)^3}{\left(p_g(X)\right)^2} > 2 p_g(X),
	$$
	except the case when $K_{S_0}^2 = 6$ and $p_g(X) = 3$ in which we have $K_X^3 \ge \frac{16}{3}$.
	
	It remains to treat the case when $K_{S_0}^2 \le 5$. In this case, by \cite[Th\'eor\`eme 6.1]{Debarre}, $q(S) \le p_g(S) \le 2$. On the other hand,  by the Leray spectral sequence, we have
	$$
	q(X) = h^2(X', \omega_{X'}) = h^1(\Sigma', R^1\psi_*\omega_{X'}) + h^0(\Sigma',\omega_{\Sigma'}) = h^1(\Sigma', R^1\psi_*\omega_{X'}).
	$$
	Since $R^1\psi_* \omega_{X'/\Sigma'}$ splits into nef line bundles \cite[Lemma 2.5]{Ohno}, we deduce that 
	$$
	q(S) = \rank R^1\psi_*\omega_{X'} \ge h^1(\Sigma', R^1\psi_*\omega_{X'}) = q(X) \ge 2.
	$$
	It follows that $p_g(S) = q(S) = q(X) = 2$. By \cite[Theorem 2.7]{Fujita} $p_g(S) = 2$ implies that $\psi_*\omega_{X'/\Sigma'}$ splits into two nef line bundles. Since $h^0(\Sigma', \psi_* \omega_{X'}) = p_g(X) \ge 3$, we deduce that $h^1(\Sigma', \psi_*\omega_{X'}) \le 1$. Furthermore, $q(S) = q(X) = 2$ implies that $R^1\psi_* \omega_{X'} = \omega_{\Sigma'}^{\oplus 2}$.
	By the Leray spectral sequence, we have
	$$
	h^2(X, \CO_X) = h^1(X', \omega_{X'}) = h^1(\Sigma', \psi_* \omega_{X'}) + h^0(\Sigma', R^1\psi_*\omega_{X'}) \le 1.
	$$
	Thus by \eqref{eq: Severi q>1}, it follows that
	$$
	K_X^3 \ge 4 \left(p_g(X) - h^2(X, \CO_X) + q(X) - 1\right) \ge 4p_g(X).
	$$
	The proof is completed.
\end{proof}

\begin{prop} \label{prop: albdim 2 canonical dim 1 case 2}
	Suppose that $\alb \dim X \ge 2$ and $p_g(X) = 2$. Then $K_X^3 \ge \frac{8}{3}$.
\end{prop}

\begin{proof}
	By \cite[Theorem 1.4]{Jiang}, if $\alb \dim X = 3$, then $K_X^3 \ge 12$. Thus to prove the proposition, we may assume that $\alb \dim X = 2$. Let $a: X \to A$ be the Albanese map of $X$, and let
	$$
	h^0_a(X, K_X) := \min \{h^0(X, \omega_X \otimes \alpha)| \alpha \in \Pic^0(X)\}.
	$$
	If $h^0_a(X, K_X) \ge 1$, by \cite[Theorem 1.3]{Zhang_Clifford}, we have
	$$
	K_X^3 \ge 4 h^0_a(X, K_X) \ge 4.
	$$
	Thus in the following, we assume that $h^0_a(X, K_X) = 0$. By the proof of \cite[Proposition 4.3]{Jiang}, under this assumption, $a$ is surjective. In particular, $A$ is an abelian surface. By the proof of \cite[Proposition 4.3]{Jiang}, we have
	$$
	a_* \omega_X = \CE_0 \oplus \left( \bigoplus_{i=1}^l \CE_i\right),
	$$
	where $\CE_0$ is either zero or a direct sum of torsion line bundles on $A$, and $\CE_i = \bigoplus_{j=1}^{n_i} (p_i^* \CF_{ij} \otimes \CP_{ij})$ is non-zero for each $i > 0$. Here $p_i: A \to E_i$ is a fibration from $A$ to a genus one curve $E_i$, each $\CF_{ij}$ is an ample vector bundle on $E_i$, each $\CP_{ij} \in \Pic^0(A)$ is a torsion line bundle and we may assume that $\CP_{ij}\notin p_i^*\Pic^0(E_i)$ when it is non-trivial. Moreover, the $p_i$'s are all distinct and $l \ge 2$.
	
	\textbf{Case 1}: $h^0(A, \CE_0) < 2$. We first consider the case when $h^0(A, \CE_0) < 2$. Since $p_g(X) = 2$, we have $h^0(A, \bigoplus_{i=1}^l \CE_i) > 0$. 
	
	Without loss of generality, we assume that $h^0(A, \CE_1) > 0$. Then at least one $\CP_{1j}$ must be trivial. Say $\CP_{11}$ is trivial. Let $f: X \to B$ be the Stein factorization of the composition $p_1 \circ a: X \to E_1$. Denote by $F$ the general fibre of $f$. Then $F$ is an irregular surface of general type. If $g(B) \ge 2$, since $K_{X/B}$ is nef, we have
	$$
	K_X^3 \ge K_{X/B}^3 + \left( 6g(B) - 6 \right)K_F^2 \ge 6K_F^2 \ge 6.
	$$
	Thus we only need to treat the case when $g(B) = 1$. In this case, by the universal property of the Albanese morphism, $p_1$ factors through the morphism $A \to B$. Thus we may further assume that $B = E_1$ and $f = p_1 \circ a$. Now $\CF_{11}$ is a direct summand of $f_* \omega_X$. Since $f_* \omega_X$ is nef \cite[Theorem 2.7]{Fujita}, we know that 
	$$
	\deg f_* \omega_X \ge \deg \CF_{11} \ge 1.
	$$
	
	Suppose that $\deg f_* \omega_X \ge 2$. By \cite[Theorem 1.6]{Hu_Zhang1}, we have
	$$
	K_X^3 \ge 2 \deg f_* \omega_X \ge 4.
	$$
	Thus in the following, we assume that $\deg f_* \omega_X = 1$. Then $\deg \CF_{11} = 1$. If $K_F^2 \ge 8$, by \cite[Theorem 1.7]{Hu_Zhang1}, we have
	$$
	K_X^3 \ge \left(\frac{4K_F^2}{K_F^2 + 4}\right) \deg f_* \omega_X = \frac{4K_F^2}{K_F^2 + 4} \ge \frac{8}{3}.
	$$
	Now assume that $K_F^2 < 8$. By \cite[Th\'eor\`eme 6.1]{Debarre}, $K_F^2 \ge 2p_g(F)$. Since $\CF_{11}$ is ample, we know that $h^1(E_1, \CF_{11}) = 0$. Thus by the Riemann-Roch theorem, $h^0(E_1, \CF_{11}) = 1$. Note that $h^0(E_1, f_* \omega_X) = p_g(X) = 2$. Thus $\rank \CF_{11} < \rank f_* \omega_X = p_g(F)$.
	On the other hand, by \cite[Corollary 3.5]{Miyaoka}, we know that $K_X - \mu(\CF_{11})F$ is pseudo-effective. We deduce that
	$$
	K_X^3 \ge \mu(\CF_{11}) K_F^2 \ge \frac{K_F^2}{p_g(F) - 1} \ge \frac{2K_F^2}{K_F^2 - 2} > \frac{8}{3}
	$$
	
	\textbf{Case 2}: $h^0(A, \CE_0) = 2$. We now treat the case when $h^0(A, \CE_0) = 2$. In this case, $\CE_0$ contains $\CO_A^{\oplus 2}$ as a direct summand. 
	
	Since $p_1$ and $p_2$ are distinct, $p_1^*\Pic^0(E_1)$ and $p_2^*\Pic^0(E_2)$ generate $\Pic^0(A)$. We may assume that each $\CP_{1j} \in p_2^*\Pic^0(E_2)$ and each $\CP_{2j} \in p_1^*\Pic^0(E_1)$. Since $h^0(A, \CE_0) = p_g(X) = 2$, every $\CP_{1j}$ and $\CP_{2j}$ is a non-trivial torsion. Applying the same argument as in Case 1, we see that the proof is reduced again to the case when both $f_i = p_i \circ a: X \to E_i$ have connected fibres for $i = 1, 2$. Now $\CF_{i1}$ is a direct summand of $(f_i)_* (\omega_X \otimes a^*\CP_{i1}^{-1})$. Since $\CP_{i1}$ is a torsion, by \cite[Corollary 3.4]{Viehweg} (take $\mathcal{L}=a^*\CP_{11}^{-1}$ and $D=0$), we know that $(f_i)_* (\omega_X \otimes a^*\CP_{i1}^{-1})$ is also nef. Thus
	$$
	\deg (f_i)_* (\omega_X \otimes a^*\CP_{i1}^{-1}) \ge \deg \CF_{i1} \ge 1.
	$$
	
	Let $F_i$ be a general fibre of $f_i$ for $i=1, 2$. If $K_{F_i}^2 \ge 8$ for any $i$, applying the same proof of \cite[Theorem 1.7]{Hu_Zhang1} for $\omega_X \otimes \CP_{i1}^{-1}$, we deduce that
	$$
    K_X^3 \ge \left(\frac{4K_{F_i}^2}{K_{F_i}^2 + 4}\right) \deg (f_i)_* (\omega_X \otimes a^*\CP_{i1}^{-1}) = \frac{4K_{F_i}^2}{K_{F_i}^2 + 4} \ge \frac{8}{3}.
    $$
    
    From now on, we assume that $K_{F_i}^2 \le 7$ for $i = 1, 2$. By the assumption on $a_* \omega_X$, we have $\CO_{E_1}^{\oplus 2} \oplus (p_1)_*\CE_2 \subset (f_1)_* \omega_X$. Note that $(p_1)_*\CE_2 \ne 0$. Thus $p_g(F_1) = \rank (f_1)_* \omega_X \ge 3$. By the proof of \cite[Proposition 4.3]{Jiang}, an \'etale cover of $F_1$ is of maximal Albanese dimension. Thus the Severi inequality $K_{F_1}^2 \ge 4 \chi(\omega_{F_1})$ \cite[Theorem 2.1]{Pardini} implies that $\chi(\omega_{F_1}) = 1$. Since $K_{F_1}^2 \ge 2p_g(F_1)$ \cite[Th\'eor\`eme 6.1]{Debarre}, we deduce that $p_g(F_1) = q(F_1) = 3$. By \cite{Catanese_Ciliberto_Lopes}, it follows that $K_{F_1}^2 = 6$ and $F_1$ is isomorphic to a smooth theta divisor in a principally polarized abelian $3$-fold $V$. The same holds also for $F_2$. 
    
    We claim that $a$ has connected fibres. Otherwise, let $h: X \to W$ be the Stein factorization of $a$. Then $W$ is a normal surface whose smooth model has Kodaira dimension at least one. Note that both $h_i: W \to E_i$ have connected fibres for $i = 1, 2$. Thus there exists an $i$ such that the general fibre of $h_i$ has genus at least two. Say $i = 1$, and let $B$ be a general fibre of $h_1$. Then $g(B) \ge 2$. Moreover, the morphism $h|_{F_1}: F_1 \to B$ is just the Stein factorization of $a|_{F_1}$. Let $C$ be a general fibre of $h|_{F_1}$. Since $K_{F_1/B}$ is nef, we have
    $$
    K_{F_1}^2 = K_{F_1/B}^2 + 8(g(B) - 1) (g(C) - 1) \ge 8(g(B) - 1) (g(C) - 1) \ge 8.
    $$
    This is absurd. Thus the claim is verified.
    
    By the claim, $a|_{F_1}: F_1 \to E'_2$ is a fibration, where $E'_2$ is a general fiber of $p_1$. It is clear that $p_2|_{E'_2}: E'_2\to E_2$ is an isogeny. Since $\CP_{11} \in p_2^*\Pic^0(E_2)$ is non-trivial, $\CP := (a^*\CP_{11}^{-1})|_{F_1}$ is a non-trivial torsion line bundle on $F$. Since $F_1$ is a smooth theta divisor in $V$ and $\Pic^0(F_1) \simeq \Pic^0(V)$, we may assume that $\CP \in \Pic^0(V)$ is non-trivial. Consider the long exact sequence
    $$
    H^0(V, \CP) \to H^0(V, \CO_V(F_1) \otimes \CP) \to H^0(F_1, \omega_{F_1} \otimes \CP) \to H^1(V, \CP).
    $$
    The non-triviality of $\CP$ implies that $h^0(V, \CP) = h^1(V, \CP) = 0$. It follows that $h^0(F_1, \omega_{F_1} \otimes \CP) = h^0(V, \CO_V(F_1) \otimes \CP) = 1$. As a result, $\rank (f_1)_* (\omega_X \otimes a^*\CP_{11}^{-1}) = 1$. In particular, $\rank \CF_{11} = 1$. Thus $\mu(\CF_{11}) \ge 1$. By \cite[Corollary 3.5]{Miyaoka} again, $K_X - F_1$ is pseudo-effective. It follows that 
    $$
    K_X^3 \ge \mu(\CF_{11}) K_{F_1}^2 \ge K_{F_1}^2 = 6.
    $$
    The whole proof is completed.
\end{proof}

\subsection{The case when $X$ is Gorenstein} We now treat the case when $X$ is Gorenstein.

\begin{prop} \label{prop: albdim 2 Gorenstein}
	Suppose that $X$ is Gorenstein. Then Theorem \ref{thm: q>1} holds.
\end{prop}

\begin{proof}
	By Proposition \ref{prop: albdim 1 q>1}, we may assume that $\alb \dim X \ge 2$. Since $X$ is Gorenstein, we know that $\chi(\omega_X) \ge 1$ \cite[\S 2.1]{Chen_Chen_Zhang}. By \eqref{eq: Severi q>1}, we have $K_X^3 \ge 4$. Thus we may further assume that $p_g(X) \ge 3$. Moreover, by Proposition \ref{prop: albdim 2 canonical dim 3}, \ref{prop: albdim 2 canonical dim 2} and \ref{prop: albdim 2 canonical dim 1 case 1}, the proof reduces to the case when $p_g(X) = 3$ and $\chi(\omega_X)=1$, and we only need to show that $K_X^3 \ge 6$.
	
	Let $\Sigma$ be the canonical image of $X$. Since $p_g(X) = 3$, we have $\dim \Sigma \le 2$. If $\dim \Sigma = 2$, by Lemma \ref{lem: h^0(2L)}, 
	$h^0(X, 2K_X) \ge 3p_g(X) - 3 = 6$. On the other hand, by the Riemann-Roch formula, we have
	$$
	h^0(X, 2K_X) = \frac{1}{2}K_X^3 + 3 \chi(\omega_X) = \frac{1}{2}K_X^3 + 3.
	$$
	Thus we deduce that $K_X^3 \ge 2h^0(X, 2K_X) - 6 \ge 6$. If $\dim \Sigma = 1$, by Proposition \ref{prop: albdim 2 canonical dim 1 case 1} and the fact that $K_X^3$ is even \cite[\S 2.2]{Chen_Chen_Zhang}, we see that $K_X^3 \ge 6$. Thus the proof is completed.
\end{proof}

\subsection{Proof of Theorem \ref{thm: q>1}} It is clear that the two inequalities in Theorem \ref{thm: q>1} follow from Proposition \ref{prop: albdim 1 q>1}, \ref{prop: albdim 2 canonical dim 3}, \ref{prop: albdim 2 canonical dim 2}, \ref{prop: albdim 2 canonical dim 1 case 1}, \ref{prop: albdim 2 canonical dim 1 case 2} and \ref{prop: albdim 2 Gorenstein}.

Suppose that $K_X^3 = \frac{4}{3}p_g(X)$. Then the above propositions imply that $\alb \dim X \ge 2$ and $p_g(X) = 2, 3$. Thus $K_X^3 \le 4$. Let $a: X \to A$ be the Albanese map of $X$. As in the proof of Proposition \ref{prop: albdim 2 canonical dim 1 case 2}, we have $\dim a(X) = 2$. Consider
$$
h^0_a(X, K_X) := \min \{h^0(X, \omega_X \otimes \alpha)| \alpha \in \Pic^0(X)\}.
$$
If $h^0_a(X, K_X) \ge 1$, since $K_X^3 \le 4$, we know that $K_X^3 = 4 h^0_a(X, K_X)$ and $A$ is an abelian surface by \cite[Theorem 1.3]{Zhang_Clifford}. If $h^0_a(X, K_X) = 0$, by the proof of Proposition \ref{prop: albdim 2 canonical dim 1 case 2}, we know that $A$ is also an abelian surface. In both cases, $q(X) = 2$. Thus the proof is completed.




\section{Irregular $3$-folds of general type with $q = 1$: Part I}

Let $X$ be a minimal $3$-fold of general type with $q(X) = 1$. Let 
$$
f: X \to B
$$
be the Albanese fibration of $X$, where $B$ is a smooth curve of genus $g(B) = 1$. Throughout this section, we always assume that the general fibre $F$ of $f$ is a $(1, 2)$-surface. The main result in this section is the following theorem.

\begin{theorem} \label{thm: q=1 1}
	Let $X$ be as above. Suppose that either $X$ is Gorenstein or $p_g(X) \ge 16$. Then
	$$
	K_X^3 \ge \frac{4}{3} p_g(X).
	$$
	Moreover, if $p_g(X) \ge 16$ and $h^2(X, \CO_X) > 0$, then 
	$$
	K_X^3 \ge \frac{4}{3} p_g(X) + \frac{1}{33}.
	$$
\end{theorem}

The rest of this section is devoted to the proof.

\subsection{General settings} \label{subsection: setting (1,2)-surface}

Suppose that $p_g(X) \ge 2$. Take a birational modification $\pi: X' \to X$ such that $X'$ is smooth and $|M| = \movable|\rounddown{\pi^*K_X} |$ is base point free.
Then we have the following commutative diagram
$$
\xymatrix{
	& & X' \ar[d]_{\pi} \ar[lld]_{f'} \ar[rr]^{\psi} \ar[drr]^{\phi_{M}} & &  \Sigma' \ar[d]^{\tau}  \\
	B & & X  \ar@{-->}[rr]_{\phi_{K_X}} \ar[ll]^f  & & \Sigma         
}
$$
where $\Sigma$ is the canonical image of $X$, $\phi_M: X' \to \Sigma$ is the morphism induced by $|M|$, $X' \stackrel {\psi}\rightarrow \Sigma' \stackrel{\tau} \rightarrow \Sigma$ is the Stein factorization of $\phi_M$, and $f' = f \circ \pi$. Denote by $F'$ a general fibre of $f'$, and write $d = (\deg \tau) \cdot (\deg \Sigma)$.

We start from the following lemma.

\begin{lemma} \label{lem: h2 not 0}
	If $h^2(X, \CO_X) > 0$, then
	$$
	f_* \omega_X = \CO_B \oplus \CO_B(L),
	$$ 
	where $L$ is an effective divisor on $B$. In particular, if $p_g(X) \ge 3$, then $h^2(X, \CO_X) = 1$. Moreover, if $p_g(X) \ge 3$, then $\Sigma$ is a surface.
\end{lemma}

\begin{proof}
	By the assumption, we have $h^1(X, \omega_{X}) > 0$. Since $h^1(F, K_F) = 0$, we know that $R^1 f_* \omega_X = 0$ by the torsion freeness of $R^1f_* \omega_X$. Thus by the Leray spectral sequence, $h^1(B, f_*\omega_X) = h^1(X, \omega_X) = h^2(X, \CO_X) > 0$.  Since $\rank f_*\omega_X = p_g(F) = 2$, by Fujita's decomposition \cite[Theorem 3.1]{Fujita}, we may write
	$$
	f_* \omega_X = \CO_B \oplus \CO_B(L),
	$$
	where $L$ is an effective divisor on $B$. If $p_g(X) \ge 3$, by the Riemann-Roch theorem, $\deg L = p_g(X) - 1 \ge 2$. Thus 
	$$
	h^2(X, \CO_X) = h^1(B, f_*\omega_X) = h^1(B, \CO_B) = 1.
	$$
	Moreover, if $p_g(X) \ge 3$, by the projection formula, we deduce that
	$$
	h^0(X, K_X-F) = h^0(B, f_*\CO_X(K_X - F)) = h^0(B, L-P) = p_g(X) - 2,
	$$
	where $P = f(F)$. Therefore, the natural restriction $H^0(X, K_X) \to H^0(F, K_F)$ is surjective. Thus $\dim \phi_{K_X}(F) = 1$. Since $h^0(B, L) \ge 2$ and $h^0(B, K_X - f^*L) > 0$, we know that $\Sigma$ is a surface.
\end{proof}


Since $F$ is a minimal $(1, 2)$-surface, $|K_F|$ has a unique base point. Therefore, the horizontal part of the base locus of the relative canonical map of $X$ with respect to $f$ is a section $\Gamma$ whose intersection $\Gamma \cap F$ with $F$ is just the base point of $|K_F|$. 


\begin{prop} \label{prop: HZ1}
	Suppose that $p_g(X) \ge 4$ and $\dim \Sigma = 2$. Then the following inequalities hold:
	\begin{itemize}
		\item [(1)] $d \ge \min \{ 2p_g(X)-4, \, p_g(X) - 1 \} = p_g(X) - 1 \ge 3$;
		\item [(2)] $\displaystyle{(K_X \cdot \Gamma) \ge \frac{1}{3} d}$
		\item [(3)] $\displaystyle{K_X^3 \ge \frac{4}{3} d }$.
	\end{itemize}
\end{prop}

\begin{proof}
	The first inequality follows from \cite[Lemma 2.3 (1)]{Hu_Zhang2} and the assumption that $p_g(X) \ge 4$ and $g(B) = 1$. The second and the third one follow from \cite[Proposition 2.5 (2)]{Hu_Zhang2}.
\end{proof}

\begin{prop} \cite[Remark 2.12]{Hu_Zhang2} \label{prop: HZ2}
	Suppose that $p_g(X) \ge 4$, $\dim \Sigma = 2$, $d = p_g(X) - 1$ and $K_X - cF$ is nef for some $c \ge 0$. Then
	$$
	K_X^3 \ge \frac{4}{3}\left(p_g(X) - 1\right) + c.
	$$
\end{prop}

\subsection{The special case} 
In this subsection, we assume that
$$
f_* \omega_X = \CO_B \oplus \CO_B(L),
$$
where $e: = \deg L = p_g(X) - 1 \ge 3$. Under this assumption, we have
$$
K_X \sim f^*L + D \equiv eF + D.
$$
where $D$ is an effective divisor. Since $X$ has terminal singularities, we may write
$$
K_{X'} = \pi^*K_X + E_{\pi},
$$
where $E_{\pi}$ is an effective $\pi$-exceptional $\QQ$-divisor. Up to a birational modification of $X'$, we may further assume that $\pi^*D + E_{\pi}$ is simple normal crossing.

\subsubsection{$X$ is Gorenstein} We first consider the case when $X$ is Gorenstein. In this case, $D$ is a Cartier divisor. Write $D' = \pi^*D$. Then we have the following proposition in which the Gorenstein assumption is essential.

\begin{prop} \cite[Theorem 3.2]{Catanese_Chen_Zhang} \label{prop: 3/10}
	For any $t > \frac{3}{10}$, we have
	$$
	h^0 \left(F', K_{F'} + \roundup{t D'|_{F'}}\right) \ge 3.
	$$
\end{prop}
\begin{lemma} \label{lem: K+D IT0}
	The vector bundle 
	$$
	\CE: = f'_*\CO_{X'} \left(K_{X'} + \roundup{ \pi^*K_X - 2F' - \frac{2}{e} D'} \right)
	$$
	on $B$ is ample with $\rank \CE \ge 3$. In particular, $\mu_{\min}(\CE) > 0$.
\end{lemma}

\begin{proof}
	Note that the $\QQ$-divisor
	$$
	\pi^*K_X - 2F' - \frac{2}{e} D' \equiv \left(1 - \frac{2}{e} \right) \pi^*K_X
	$$
	is nef and big with simple normal crossing fractional part. By Corollary \ref{coro: vector bundle ample}, $\CE$ is ample. Thus $\mu_{\min}(\CE) > 0$. Note that 
	\begin{align*}
		\rank \CE & = h^0 \left(F', K_{F'} + \left.\roundup{ \pi^*K_X - 2F' - \frac{2}{e} D'} \right|_{F'} \right) \\
		& \ge h^0 \left(F', K_{F'} + \roundup{\left(1 - \frac{2}{e} \right) D'|_{F'}} \right)
	\end{align*}
    Since $1 - \frac{2}{e} \ge \frac{1}{3} > \frac{3}{10}$, by Proposition \ref{prop: 3/10}, we conclude that $\rank \CE \ge 3$.
\end{proof}

As a simple corollary of Lemma \ref{lem: K+D IT0}, we have

\begin{coro} \label{coro: K+D IT0}
	Let 
	$$
	\CF = f'_*\CO_{X'} \left(K_{X'} + \roundup{ \pi^*K_X - \frac{2}{e} D'} \right).
	$$
	Then $\rank \CF \ge 3$ and $\mu_{\min}(\CF) > 2$.
\end{coro}

\begin{lemma} \label{lem: 2K-F nef}
	One of the following two cases occurs:
	\begin{itemize}
		\item [(1)] We have 
		$
		\mu_{\min} (f_* \CO_X(2K_X)) \ge 2.
		$
		In this case, $K_X - F$ is nef.
		
		\item [(2)] We have 
		$
		f_* \CO_X(2K_X) = \CV_1 \oplus \CL_1,
		$
		where $\CV_1$ is of rank three with $\mu_{\min}(\CV_1) > 2$ and $\CL_1$ is a line bundle of degree one. In this case, $2K_X - F$ is nef.
	\end{itemize}
\end{lemma}

\begin{proof}
	Since $f'_*$ is left exact, we have the following injection
	$$
	\CF := f'_*\CO_{X'} \left(K_{X'} + \roundup{ \pi^*K_X  - \frac{2}{e} D'} \right) \to f'_*\CO_{X'}(2K_{X'}) = f_* \CO_X(2K_X).
	$$
	By Corollary \ref{coro: K+D IT0}, $\rank \CF \ge 3$ and $\mu_{\min}(\CF) > 2$. Let $\CV_1$ be the saturation of $\CF$ in $f_*\CO_X(2K_X)$. Then  $\CV_1$ is a subbundle of $f_*\CO_X(2K_X)$  with $\rank \CV_1 \ge 3$ and $\mu_{\min}(\CV_1) > 2$.
	
	Consider the evaluation morphism
	$$
	\mathrm{ev}_2: f^*f_* \CO_X(2K_X) \to \CO_X(2K_X).
	$$
	Let $\Delta_2 = \mathrm{Sing}(X) \cup \mathrm{Supp}(\CO_X(2K_X)/\mathrm{Im}(\ev_2))$.
	Since $F$ is a minimal $(1, 2)$-surface, $|2K_F|$ is base point free. We deduce that $\Delta_2$ is $f$-vertical. 
	
	Suppose that $\mu_{\min} (f_* \CO_X(2K_X)) \ge 2$. By Lemma \ref{lem: nefness}, $2K_X - 2F$ is nef. Thus (1) is proved. Suppose that $\mu_{\min} (f_* \CO_X(2K_X)) < 2$. Since $\rank f_* \CO_X(2K_X) = h^0(F, 2K_F) = 4$, we deduce that $\rank \CV_1 = 3$, and $f_* \CO_X(2K_X)$ has a direct summand line bundle $\CL_1$ with $\deg \CL_1 < 2$. Write
	$$
	f_*\CO_X(2K_X) = \CE_1 \oplus \CL_1.
	$$
	Since $\mu_{\min}(\CV_1) > 2$, by Lemma \ref{lem: vector bundle hom}, $\Hom(\CV_1, \CL_1) = 0$. It follows that $\CV_1\subset \CE_1$. Since $\CV_1$ is a subbundle and $\rank \CE_1 = 3$, we deduce that $\CV_1 = \CE_1$. By Corollary \ref{coro: vector bundle ample}, $f_* \CO_X(2K_X) = f'_* \CO_{X'}(K_{X'} + \pi^*K_X)$ is ample. As a result, $\deg \CL_1 = 1$. Thus $\mu_{\min}(f_* \CO_X(2K_X)) = 1$. By Lemma \ref{lem: nefness} again, we deduce that $2K_X - F$ is nef. This proves (2).
\end{proof}

\begin{lemma} \label{lem: 3K-2F nef}
	We have 
	$
	\mu_{\min}(f_* \CO_X(3K_X)) \ge 2.
	$
	Moreover, $3K_X - 2F$ is nef.
\end{lemma}

\begin{proof}
	Given that $\mu_{\min}(f_* \CO_X(3K_X)) \ge 2$, consider the evaluation morphism
	$$
	\mathrm{ev}_3: f^*f_*\CO_X(3K_X) \to \CO_X(3K_X),
	$$
	and let $\Delta_3 = \mathrm{Sing}(X) \cup \mathrm{Supp}(\CO_X(3K_X)/\mathrm{Im}(\ev_3))$. Since $K_X$ is nef, by Lemma \ref{lem: nefness}, for any curve $C$ not contained in the $f$-horizontal part of $\Delta_3$, we have
	$$
	\left((3K_X-2F) \cdot C \right) \ge 0.
	$$
	On the other hand, since $\baselocus|3K_F| = \baselocus|K_F|$, we see that the $f$-horizontal part of $\Delta_3$ is just the section $\Gamma$. By Proposition \ref{prop: HZ1} (1) and (2), we have
	$$
	\left((3K_X-2F) \cdot \Gamma \right) = 3(K_X \cdot \Gamma) - 2 \ge d - 2 > 0.
	$$
	As a result, $3K_X - 2F$ is nef.
	
	To prove that $\mu_{\min}(f_* \CO_X(3K_X)) \ge 2$, we may write
	$$
	f_* \CO_X(3K_X) = \CV_2 \oplus \CV_3,
	$$
	where $\mu_{\min}(\CV_2) \ge 2$  and  $\mu_{\max}(\CV_3)<2$. Since
	$
	H^0(F, K_F) \otimes H^0(F, 2K_F) \to H^0(F, 3K_F)
	$
	is surjective, we obtain a generically surjective morphism
	$$
	f_* \omega_X \otimes f_*\CO_X(2K_X) \to f_* \CO_X(3K_X).
	$$
	Recall that 
	$
	f_* \omega_X = \CO_B \oplus \CL
	$ 
	with $\deg \CL \ge 3$. If $\mu_{\min}(f_*\CO_X(2K_X)) \ge 2$, then $\mu_{\min}(f_* \CO_X(3K_X)) \ge 2$ by Lemma \ref{lem: vector bundle surjection}. By Lemma \ref{lem: 2K-F nef}, the proof reduces to the case when
	$$
	f_* \CO_X(2K_X) = \CV_1 \oplus \CL_1,
	$$
	where $\CV_1$ is of rank three with $\mu_{\min}(\CV_1) > 2$ and $\CL_1$ is a line bundle with $\deg \CL_1 = 1$. In this case, the above generic surjection becomes
	$$
	\CV_1 \oplus (\CV_1 \otimes \CL) \oplus \CL_1 \oplus (\CL_1 \otimes \CL) \to \CV_2 \oplus \CV_3.
	$$
	By Lemma \ref{lem: vector bundle hom}, $\Hom(\CV_1, \CV_3) = \Hom(\CL_1 \otimes \CL, \CV_3) = \Hom(\CV_1\otimes\CL, \CV_3) = 0$. By the generic surjectivity, either $\CV_3= 0$ or $\CV_3$ is a line bundle with $\deg \CV_3 = 1$. However, the latter cannot occur. Indeed, by Lemma \ref{lem: 2K-F nef}, $2K_X - F$ is nef. By Proposition \ref{prop: HZ1} (1) and (3), we have
	$$
	(2K_X - F)^3 = 8K_X^3 - 6K_F^2 = 8K_X^3 - 6 > 0.
	$$
	Thus $2K_X - F$ is also big. By Corollary \ref{coro: vector bundle ample}, $f_*\CO_X(3K_X - F) = f'_*\CO_{X'}(K_{X'} + \pi^*(2K_X - F))$ is ample. In particular, $\mu_{\min} (f_* \CO_X(3K_X)) > 1$. Thus the proof is completed.
\end{proof}

\begin{prop} \label{prop: 4K-3F nef}
	We have 
	$
	\mu_{\min}(f_* \CO_X(4K_X)) \ge 3.
	$
	Moreover, $4K_X - 3F$ is nef.
\end{prop}

\begin{proof}
	Given that $\mu_{\min}(f_* \CO_X(4K_X)) \ge 3$, consider the evaluation map
	$$
	\mathrm{ev}_4: f^*f_* \CO_X(4K_X) \to \CO_X(4K_X),
	$$
	and let $\Delta_4 = \mathrm{Sing}(X) \cup \mathrm{Supp}(\CO_X(4K_X)/\mathrm{Im}(\ev_4))$. 
	Since $|4K_F|$ is base point free, we deduce that $\Delta_4$ is $f$-vertical. By Lemma \ref{lem: nefness}, $4K_X - 3F$ is nef.
	
	To prove that $\mu_{\min}(f_* \CO_X(4K_X)) \ge 3$, we may write
	$$
	f_* \CO_X(4K_X) = \CV_4 \oplus \CV_5,
	$$
	where $\mu_{\min}(\CV_4) \ge 3$ and $\mu_{\max}(\CV_5) < 3$. Since
	$
	\Sym^2 H^0(F, 2K_F) \to H^0(F, 2K_F)
	$
	is surjective, we obtain a generically surjective morphism
	$$
	\Sym^2 f_*\CO_X(2K_X) \to f_* \CO_X(4K_X).
	$$
	Similar to the proof of Lemma \ref{lem: 3K-2F nef}, the proof reduces again to the case when
	$$
	f_* \CO_X(2K_X) = \CV_1 \oplus \CL_1,
	$$
	where $\CV_1$ is of rank three with $\mu_{\min}(\CV_1) > 2$ and $\CL_1$ is a line bundle with $\deg \CL_1 = 1$. Now the above generic surjection becomes
	$$
	\Sym^2 \CV_1 \oplus (\CV_1 \otimes \CL_1) \oplus \CL_1^{\otimes 2} \to \CV_4 \oplus \CV_5.
	$$
	By Lemma \ref{lem: vector bundle tensor} and \ref{lem: vector bundle hom}, $\Hom(\Sym^2 \CV_1, \CV_5) = \Hom(\CV_1 \otimes \CL_1, \CV_5) = 0$. Thus the generic surjectivity implies that either $\CV_5= 0$ or $\CV_5$ is a line bundle with $\deg \CV_5 = 2$. On the other hand, by Lemma \ref{lem: 3K-2F nef}, $3K_X - 2F$ is nef. By Proposition \ref{prop: HZ1} (1) and (3), we have
	$$
	(3K_X - 2F)^3 = 27K_X^3 - 54K_F^2 = 27K_X^3 - 54 > 0.
	$$
	Thus $3K_X - 2F$ is also big. By Corollary \ref{coro: vector bundle ample},  $f_*\CO_X(4K_X - 2F) = f'_* \CO_{X'}(K_{X'} + \pi^*(3K_X - 2F))$ is ample. In particular, $\mu_{\min} (f_* \CO_X(4K)) > 2$. We conclude that $\CV_5 = 0$, and the proof is completed.
\end{proof}

\subsubsection{General case} Now we consider the general case, without assuming $X$ is Gorenstein. We first recall a result due to Koll\'ar.

\begin{prop} \cite[Theorem A.1]{Chen_Chen_Jiang} \label{prop: lct}
	Let $\Delta \ge 0$ be a $\QQ$-divisor on $F$ such that $\Delta \sim_{\QQ} K_F$. Then
	$$
	\mathrm{lct}(F; \Delta) \ge \frac{1}{10}.
	$$
\end{prop}

\begin{lemma} \label{lem: K-F nef}
	If $e \ge 11$, then $K_X - F$ is nef and big.
\end{lemma}

\begin{proof}
	Since the nef cone of $X$ is closed, it is enough to prove that $nK_X - (n-1)F$ is nef for any $n \ge 1$. As long as $K_X - F$ is nef, we deduce that 
	$$
	(K_X - F)^3 = K_X^3 - 3K_F^2 = K_X^3 - 3 > 0
	$$
	by Proposition \ref{prop: HZ1} (1) and (3). Thus $K_X - F$ is big.
	
	We prove by induction on $n$. The case $n=1$ is clear. Suppose that $mK_X - (m-1)F$ is nef. Note that the $\QQ$-divisor
	$$
	m\pi^*(K_X - F) -  \frac{m}{m(e-1) + 1} D' \equiv \left(1 - \frac{1}{m(e-1) + 1}\right) \left(m\pi^*K_X - (m-1)F'\right)
	$$
	is nef and big with a simple normal crossing fractional part. Let 
	$$
	\CE_m = f'_* \CO_{X'}\left(K_{X'} + \roundup{m\pi^*(K_X - F) - \frac{m}{m(e-1) + 1} D'} \right).
	$$
	By Corollary \ref{coro: vector bundle ample}, $\CE_m$ is ample. On the other hand, since $e \ge 11$, for any $m > 0$, we have
	$$
	\frac{m}{m(e-1) + 1} < \frac{1}{10}.
	$$ 
	By Proposition \ref{prop: lct}, the pair $(F, \frac{m}{m(e-1) + 1}D|_F)$ is klt. In particular,
	$$
	(\pi|_{F'})_* \CO_{F'} \left(K_{F'/F} - \rounddown{\frac{m}{m(e-1) + 1} D'|_{F'}} \right) = \CO_F.
	$$
	It follows that
	\begin{align*}
		\rank \CE_m & = h^0 \left. \left(F', K_{F'} + \roundup{m\pi^*(K_X - F)- \frac{m}{m(e-1) + 1} D'}  \right|_{F'} \right) \\
		& = h^0 \left(F', (m+1) (\pi|_{F'})^*K_F + K_{F'/F} - \rounddown{\frac{m}{m(e-1) + 1} D'|_{F'}}\right) \\
		& = h^0\left(F, (m+1)K_F\right).
	\end{align*}
	Since there is an injection $\CE_m \hookrightarrow f_* \CO_X\left((m+1)K_X - mF\right)$ of vector bundles of the same rank, we deduce that $f_* \CO_X\left((m+1)K_X - mF\right)$ is ample. Let 
	$$
	\ev_{m+1}: f^*f_* \CO_X \left((m+1)K_X - mF\right) \to \CO_X \left((m+1)K_X - mF\right)
	$$
	be the evaluation morphism, and let 
	$$
	\Delta_{m+1} = \mathrm{Sing}(X) \cup \mathrm{Supp}(\CO_X((m+1)K_X - mF)/\mathrm{Im}(\ev_{m+1})).
	$$ 
	By Lemma \ref{lem: nefness}, $(m+1)K_X - mF$ is nef away from the $f$-horizontal part of $\Delta_{m+1}$. Note that $F$ is a $(1, 2)$-surface. If the $f$-horizontal part of $\Delta_{m+1}$ is non-empty, then $m = 2$ and the $f$-horizontal part of $\Delta_{m+1}$ is the section $\Gamma$. Nevertheless, by Proposition \ref{prop: HZ1} (1) and (2), we have
	$$
	\left((3K_X - 2F) \cdot \Gamma \right) = 3(K_X \cdot \Gamma) - 2 \ge d - 2 > 0.
	$$
	Thus the induction process works, and the result follows. 
\end{proof}

\begin{lemma} \label{lem: more nefness}
	Suppose that $e \ge 11$ and that $K_X - aF$ is nef and big for some rational number $a \ge 1$. Then $K_X - \frac{9a + e}{20}F$ is nef and big.
\end{lemma}

\begin{proof}
	We may write $a = \frac{p}{q}$ with $(p, q) = 1$. For any $n > 1$, let $\nu_n: B_n \to B$ be an \'etale cover of degree $nq$. Then $g(B_n) = 1$. Let $X_n = X \times_B B_n$. Then we have the following commutative diagram
	$$
	\xymatrix{
		X_n \ar[d]_{f_n} \ar[rr]^{\mu_n} & &  X \ar[d]^{f}  \\
		B_n \ar[rr]^{\nu_n} & & B
	}
    $$
    Denote by $F_n$ a general fibre of $f_n: X_n \to B_n$. Then we have
    $$
    K_{X_n} = \mu_n^*K_X \equiv nqeF_n + D_n,
    $$
    where $D_n = \mu_n^*D$. Let $\pi_n: X'_n \to X_n$ be birational modification as in \S \ref{subsection: notation} with respect to $|K_{X_n}|$. Since $X_n$ has terminal singularities, we may write
    $$
    K_{X'_n} = \pi_n^*K_{X_n} + E_{\pi_n}, 
    $$
    where $E_{\pi_n} \ge 0$ is a $\QQ$-divisor. We may further assume that $\pi_n^*D_n + E_{\pi_n}$ is simple normal crossing. By the assumption, 
    $$
    (qe - p)nF_n + D_n \equiv K_{X_n} - npF_n \equiv \mu_n^*(K_X - aF)
    $$
    Set $m = \roundup{\frac{(qe - p)n}{10}} - 1$. Then the $\QQ$-divisor
    $$
    \pi_n^* \left(K_{X_n} - (np + m)F_n - \frac{m}{(qe - p)n} D_n\right) \equiv \left(1 - \frac{m}{(qe - p)n}\right)\pi_n^*(K_{X_n} - npF_n)
    $$
    is nef and big with a simple normal crossing fractional part. Note that
    $$
    \frac{m}{(qe - p)n} < \frac{1}{10}.
    $$
    Thus we may run a similar argument as in the proof of Lemma \ref{lem: K-F nef} to deduce that $2K_{X_n} - (np + m) F_n$ is nef, which implies that $K_X - \left(\frac{a}{2} + \frac{m}{2nq} \right) F$ is nef. By the choice of $m$, we have 
    $$
    \lim\limits_{n \to + \infty} \frac{m}{nq} = \frac{e}{10} - \frac{q}{10 p} = \frac{e-a}{10}.
    $$
    Taking $n \to \infty$, it follows that $K_X - \frac{9a + e}{20}F$ is nef. Note that $K_X - aF$ is nef implies that $K_X^3 \ge 3aK_F^2 = 3a$, and Proposition \ref{prop: HZ1} (3) implies that $K_X^3 \ge \frac{4}{3}e$. It follows that
    $$
    \left(K_X - \frac{9a + e}{20}F\right)^3 = K_X^3 - \frac{27a + 3e}{20} \ge \frac{3a}{2} + \frac{2e}{3} - \frac{27a +3e}{20} > 0.
    $$
    Thus $K_X - \frac{9a + e}{20}F$ is also big, and the proof is completed.
\end{proof}

\begin{prop} \label{prop: final nefness}
	Suppose that $e \ge 11$. Then $K_X - \frac{e}{11}F$ is nef.
\end{prop}

\begin{proof}
	Since $e \ge 11$, by Lemma \ref{lem: K-F nef},  $K_X-F$ is nef. Take $a_1=1$ and define $a_{n+1} = \frac{9a_n+e}{20}$ recursively. Then it is easy to see that $\{a_n\}$ is an increasing sequence and bounded from above. Moreover, we have
	$$
	\lim\limits_{n \to \infty} a_n = \frac{e}{11}.
	$$
	By Lemma \ref{lem: more nefness}, $K_X-a_nF$ is nef for every $n \ge 2$. Thus $K_X - \frac{e}{11}F$ is nef.
\end{proof}

\subsection{Proof of Theorem \ref{thm: q=1 1}} Now we prove Theorem \ref{thm: q=1 1}. By \cite[Theorem 1.1]{Hu_Zhang2}, we have
$$
K_X^3 \ge \frac{4}{3} \chi(\omega_X) = \frac{4}{3}p_g(X) - \frac{4}{3}h^2(X, \CO_X).
$$
If $h^2(X, \CO_X) = 0$, then there is nothing to prove. Thus in the following, we assume that $h^2(X, \CO_X) > 0$.

We first treat the case when $p_g(X) \ge 4$. By Lemma \ref{lem: h2 not 0}, we know that $\dim \Sigma =2$. If $d \ge p_g(X)$, then the theorem follows from Proposition \ref{prop: HZ1} (3). Thus we only need to consider the case when $d = p_g(X) - 1$. If $p_g(X) \ge 16$. by Proposition \ref{prop: HZ2} and \ref{prop: final nefness}, we have
$$
K_X^3 \ge \frac{4}{3} \left(p_g(X) - 1\right) + \frac{15}{11} =  \frac{4}{3}p_g(X) + \frac{1}{33}.
$$
If $X$ is Gorenstein, by Proposition \ref{prop: HZ2} and \ref{prop: 4K-3F nef}, we have
$$
K_X^3 \ge \frac{4}{3} \left(p_g(X) - 1\right) + \frac{3}{4} = \frac{4}{3}p_g(X) - \frac{7}{12},
$$
i.e., $3K_X^3\ge 4p_g(X)-\frac{7}{4}$. Since now $K_X^3 > 0$ is an even integer \cite[\S 2.2]{Chen_Chen_Zhang}, we deduce that
$$
3K^3_X \ge 4p_g(X).
$$

We are left to treat the case when $X$ is Gorenstein and $p_g(X) \le 3$. Since $K_X^3 > 0$ is an even integer, the theorem is trivial when $p_g(X) \le 1$. We may assume that $2 \le p_g(X) \le 3$, and we only need to prove that $K_X^3 \ne 2$. By Lemma \ref{lem: h2 not 0}, we have
$$
f_* \omega_X = \CO_B \oplus \CL,
$$ 
where $\CL$ is a line bundle on $B$ with $\deg \CL = p_g(X) - 1$. Consider the following commutative diagram
$$
\xymatrix{
	Y \ar[d]_{f'} \ar[rr] & &  X \ar[d]^{f}  \\
	C \ar[rr]^{\nu} & & B
}
$$
where $\nu: C \to B$ is an \'etale cover of degree three, and $Y = X \times_B C$. It is easy to see that $g(C) = 1$ and that $Y$ is Gorenstein minimal. Moreover, $f': Y \to C$ is the Albanese fibration of $Y$ with general fibre isomorphic to $F$. Note that 
$$
f'_*\omega_Y = \nu^*f_* \omega_X =  \CO_C \oplus \nu^*\CL.
$$
Thus $\deg \nu^*\CL = 3 \deg \CL = 3p_g(X) - 3$ and $p_g(Y) = h^0(C, f'_*\omega_Y) = 3p_g(X) - 2$.

Suppose that $p_g(X) = 3$. Then $p_g(Y) = 7$. Thus $K_Y^3 \ge \frac{4}{3} p_g(Y) = \frac{28}{3}$, which implies that
$$
K_X^3 = \frac{1}{3} K^3_Y \ge \frac{28}{9} > 3.
$$
Suppose that $p_g(X) = 2$. Then $\deg \CL = 1$ and $p_g(Y) = 4$. Since $X$ is Gorenstein, by \cite[\S 2.1]{Chen_Chen_Zhang}, we know that 
$$
p_g(X) + q(X)- h^2(X, \CO_X) - 1 = \chi(\omega_X) > 0.
$$
Thus $h^2(X, \CO_X) = 1$ and $\chi(\omega_X) = 1$. Since $\Sym^2 H^0(F, K_F) \to H^0(F, 2K_F)$ is an injection, pushing forward by $f$, we obtain an injection
$$
\CO_B \oplus \CL \oplus \CL^{\otimes 2} = \Sym^2 f_* \omega_X \to f_*\CO_X(2K_X).
$$
In particular, $f_*\CO_X(2K_X)$ admits an indecomposable direct summand $\CV_2$, contained in the image of $\CL^{\otimes 2}$, such that $\mu(\CV_2) \ge \mu(\CL^{\otimes 2}) = 2$. Suppose on the contrary that $K_X^3 = 2$. By the Riemann-Roch formula, we have 
$$
h^0 \left(B, f_*\CO_X(2K_X)\right) = h^0(X, 2K_X) = \frac{1}{2}K_X^3 + 3 \chi(\omega_X) = 4.
$$
By Corollary \ref{coro: vector bundle ample}, $f_*\CO_X(2K_X)$ is ample. We deduce that $\mu_{\min}(f_*\CO_X(2K_X)) > 0$ and $\deg f_*\CO_X(2K_X) = 4$. Moreover, $\rank f_*\CO_X(2K_X) = h^0(F, 2K_F) = 4$. This forces $\rank \CV_2 = 1$. Together with the above injection, it implies that $f_*\CO_X(2K_X)$ admits another indecomposable direct summand $\CV_2' \ne \CV_2$, contained in the image of $\CL$, such that $\mu(\CV_2') \ge \mu(\CL) = 1$. The existence of such $\CV_2$ and $\CV'_2$ in turn forces
$$
f_*\CO_X(2K_X) = \CV_2 \oplus \CV'_2 \oplus \CE,
$$
where $\deg \CV_2 = 2$, $\CV_2'$ is a line bundle with $\deg \CV_2' = 1$, and $\CE$ is an indecomposable bundle of rank two with $\deg \CE = 1$. In particular, we have $\mu_{\min} \left(f_*\CO_X(2K_X)\right) = \frac{1}{2}$. Since $f'_* \CO_Y(2K_Y) = \nu^*f_*\CO_X(2K_X)$, we deduce that $\mu_{\min} \left(f'_* \CO_Y(2K_Y)\right) = \frac{3}{2}$. However, this cannon occur by Lemma \ref{lem: 2K-F nef}. As a result, we deduce that $K_X^3 \ne 2$. Thus the proof is completed.

\section{Irregular $3$-folds of general type with $q = 1$: Part II} \label{section: q=1 II}

Let $X$ be a minimal $3$-fold of general type with $q(X) = 1$. Let 
$$
f: X \to B
$$
be the Albanese fibration of $X$, where $B$ is a smooth curve of genus $g(B) = 1$. Throughout this section, we always assume that the general fibre $F$ of $f$ is not a $(1, 2)$-surface. The main result in this section is the following theorem.

\begin{theorem} \label{thm: q=1 2}
	Let $X$ be as above. Suppose that $X$ is Gorenstein or $p_g(X) \ge 6$. Then
	$$
	K_X^3 \ge \frac{4}{3} p_g(X),
	$$
	and the equality holds only when $X$ is Gorenstein and $p_g(X) = 3$.
\end{theorem}

Note that the theorem is trivial when $X$ is Gorenstein and $p_g(X) \le 1$, because under this assumption we always have $K_X^3 \ge 2$ \cite[\S 2.2]{Chen_Chen_Zhang}. Thus in the following, we always assume that $p_g(X) \ge 2$. As in \S \ref{subsection: setting (1,2)-surface}, we still have the following commutative diagram:
$$
\xymatrix{
	& & X' \ar[d]_{\pi} \ar[lld]_{f'} \ar[rr]^{\psi} \ar[drr]^{\phi_{M}} & &  \Sigma' \ar[d]^{\tau}  \\
	B & & X  \ar@{-->}[rr]_{\phi_{K_X}} \ar[ll]^f  & & \Sigma         
}
$$
We still denote by $|M|$ the movable part of $|\rounddown{\pi^*K_X}|$ and by $F'$ a general fibre of $f'$. Note that by \cite[Theorem 1.6 and Corollary 1.9]{Zhang_Clifford}, $X$ satisfies the Severi inequality
\begin{equation} \label{eq: Severi q=1}
	K_X^3 \ge 2\chi(\omega_X)
\end{equation}
and the slope inequality
\begin{equation} \label{eq: slope q=1}
	K_X^3 \ge 2 \deg f_* \omega_X.
\end{equation}

\subsection{The case when $\dim \Sigma = 1$} We start with the following proposition.

\begin{prop} \label{prop: q=1 canonical dim 1}
	Suppose that $\dim \Sigma = 1$. 
	\begin{itemize}
		\item [(1)] If $p_g(X) = 3$, then $K_X^3 \ge \frac{32}{9}$.
		\item [(2)] If $p_g(X) = 4$, then $K_X^3 \ge \frac{27}{4}$;
		\item [(3)] If $p_g(X) \ge 5$, then $K_X^3 \ge 2p_g(X)$.
	\end{itemize}
\end{prop}

\begin{proof}
	As in the proof of Proposition \ref{prop: albdim 2 canonical dim 1 case 1}, we may write 
	$
	M \equiv bS,
	$
	where $S$ is a general fibre of $\psi$. Let $\sigma: S \to S_0$ be the contraction onto the minimal model $S_0$ of $S$.
	
    We first consider the case when $g(\Sigma') > 0$. In this case, since $q(X) = 1$, we have $g(\Sigma') = 1$. Thus the universal property of the Albanese map implies that $\psi = f'$, and we may assume that $S = F'$ and $\Sigma' = B$. Moreover, $g(\Sigma') = 1$ implies that $b = p_g(X)$. Thus it follows that
	$$
	K_X^3 = (\pi^*K_X)^3 \ge b \left((\pi^*K_X)^2 \cdot F'\right) = p_g(X) K_F^2.
	$$
	If $K_F^2 \ge 2$, then the above inequalities implies that
	$$
	K_X^3 \ge 2p_g(X).
	$$ 
	If $K_F^2 = 1$, since $F$ is not a $(1, 2)$-surface and $p_g(X) > 0$, we deduce that $p_g(F) = 1$. Now $f_* \omega_X$ is a line bundle on $B$, and $h^0(B, f_* \omega_X) = p_g(X) \ge 2$. Thus $h^1(B, f_* \omega_X) = 0$. By \eqref{eq: slope q=1} and the Riemann-Roch theorem, we have
	$$
	K_X^3 \ge 2 \deg f_* \omega_X = 2p_g(X).
	$$
	
	We now assume that $g(\Sigma') = 0$. In this case, $f'|_S: S \to B$ is non-trivial. Thus $q(S) > 0$. Suppose first that $K_{S_0}^2 \ge 4$. As in the proof of Proposition \ref{prop: albdim 2 canonical dim 1 case 1}, we have
	$$
	K_X^3 \ge \frac{\left(p_g(X) - 1\right)^3}{\left(p_g(X)\right)^2} K_{S_0}^2 \ge \frac{4\left(p_g(X) - 1\right)^3}{\left(p_g(X)\right)^2},
	$$ 
	and it is easy to check that the three inequalities in the proposition hold. It remains to treat the case when $K_{S_0}^2 \le 3$. By \cite[Th\'eor\`eme 6.1]{Debarre}, $K_{S_0}^2 \ge 2p_g(S) \ge 2q(S)$. This forces $p_g(S) = q(S) = 1$. First, $p_g(S) = 1$ implies that $ \psi_* \omega_{X'}$ is a line bundle, and $p_g(X) \ge 2$ further implies that $\psi_*\omega_{X'}$ is ample. Second, $q(S) = 1$ implies that $R^1\psi_*\omega_{X'}$ is a line bundle. By the Leray spectral sequence, $h^1(\Sigma', R^1\psi_*\omega_{X'}) = h^2(X', \omega_{X'}) - g(\Sigma') = 1$. Thus $R^1\psi_*\omega_{X'} = \omega_{\Sigma'}$. By the Leray spectral sequence again, we deduce that
	$$
	h^2(X, \CO_X) =h^1(X',\omega_{X'})= h^1(\Sigma',  \psi_* \omega_{X'}) + h^0(\Sigma', R^1 \psi_* \omega_{X'}) = 0.
	$$
	By \eqref{eq: Severi q=1}, it follows that
	$$
	K_X^3 \ge 2\chi(\omega_X) = 2p_g(X) + 2q(X) - 2 = 2p_g(X).
	$$
	The whole proof is completed.
\end{proof}

\subsection{The case when $\dim \Sigma = 2$} We now consider the case when $\dim \Sigma = 2$. Thus $p_g(X) \ge 3$. Since $\dim \phi_{K_X}(F) \ge 1$, we have $p_g(F) \ge 2$. We still denote by $C$ a general fibre of $\psi$, where $C$ is a smooth curve of genus $g(C) \ge 2$. 

As in \S \ref{subsubsection: albdim 2 canonical dim 2}, we may assume that $X'$ and $\Sigma'$ are smooth. Let $S \in |M|$ be a smooth general member. Denote by $\sigma: S \to S_0$ the contraction onto its minimal model. Then we have $M|_S \equiv d C$, where $d = (\deg \tau) \cdot (\deg \Sigma) \ge p_g(X) - 2$. Let $H'$ be a smooth curve on $\Sigma'$ such that $S = \psi^*H'$.

\begin{lemma} \label{lem: q=1 canonical dim 2 case 1}
	Suppose that $\dim \Sigma = 2$ and $C$ is $f'$-vertical. Then $K_X^3 \ge 2p_g(X) - 2$.
\end{lemma}

\begin{proof}
	By the assumption, $C$ is contained in a general fibre of $f'$. Thus we have
	$$
	K_X^3 \ge \left( (\pi^*K_X) \cdot M^2 \right) = d \left( (\pi^*K_X)|_{F'} \cdot C\right) = d (K_F \cdot C_0).
	$$
	where $C_0 = \pi_*C$. Note that $(K_F \cdot C_0) \ge 2$. Otherwise, we would have $(K_F \cdot C_0) = 1$. By the Hodge index theorem, we deduce that $K_F^2 = 1$, and $F$ is a $(1, 2)$-surface, which is absurd. Now $f'$ factors through $\psi$. It implies that $\Sigma'$ is an irregular surface. By \cite[Theorem 7]{Nagata}, we know that $d \ge p_g(X) - 1$. Combine the above inequalities together, and it follows that
	$$
	K_X^3 \ge 2p_g(X) - 2. 
	$$
	Thus the proof is completed.
\end{proof}

\begin{lemma} \label{lem: q=1 canonical dim 2 case 2}
	Suppose that $\dim \Sigma = 2$ and $C$ is not $f'$-vertical. 
	\begin{itemize}
		\item [(1)] If $g(H') \ge 1$, then $K_X^3 \ge 2p_g(X) - 2$;
		
		\item [(2)] If $g(H') = 0$, then
		$$
		K_X^3 \ge \min\left\{\frac{5}{2}p_g(X) - 5,  2p_g(X) - \frac{10}{3}\right\}.
		$$
	\end{itemize}
\end{lemma}

\begin{proof}
	Since $S$ admits a surjective morphism $f'|_S: S \to B$, we know that $q(S) > 0$. By the assumption, $f'|_S$ does not factor through $\psi|_S: S \to H'$. 
	
	If $g(H') \ge 1$, then we deduce that $\alb \dim S = 2$. By the same argument as in the proof of Lemma \ref{lem: albdim 2 canonical dim 2 case 2}, we deduce that $K_X^3 \ge 2d$. By \cite[Theorem 7]{Nagata}, we have $d \ge p_g(X) - 1$. Thus it follows that 
	$$
	K_X^3 \ge 2p_g(X) - 2.
	$$
	The proof of (1) is completed. 
	
	In the following, we assume that $g(H') = 0$.
	By \cite[Corollary 2.3]{Chen_Chen_Jiang}, the $\QQ$-divisor $2(\pi^*K_X)|_S - \sigma^*K_{S_0}$ is pseudo-effective. Thus we have
	$$
	2\left( (\pi^*K_X) \cdot C \right) \ge \left((\sigma^*K_{S_0}) \cdot C \right) = (K_{S_0} \cdot C_0),
	$$
	where $C_0 = \sigma_* C$. If $(K_{S_0} \cdot C_0) \ge 5$, then $\left( (\pi^*K_X) \cdot C \right) \ge \frac{5}{2}$ and we have
	$$
	K_X^3 \ge \left( (\pi^*K_X) \cdot M^2 \right) = d \left( (\pi^*K_X) \cdot C \right) \ge \frac{5}{2}p_g(X) - 5.
	$$
	Thus we may assume that $(K_{S_0} \cdot C_0) \le 4$. Note that by the adjunction, $K_S - 2M|_S \equiv K_S - 2dC$ is pseudo-effective. Thus $K_{S_0} - 2dC_0$ is pseudo-effective. It follows that
	$$
	4 \ge (K_{S_0} \cdot C_0) \ge 2d C_0^2 \ge 2C_0^2 \left(p_g(X) - 2\right) \ge 2C_0^2.
	$$
	In particular, $C_0^2 = 0$, $1$ or $2$. We claim that in all cases, 
	$$
	K_{S_0}^2 \ge 4 \chi(\omega_S) - 8.
	$$ 
	Suppose the claim holds. Then as in the proof of Lemma \ref{lem: albdim 2 canonical dim 2 case 1}, we know that
	$$
	\chi(\omega_X) + \chi(\omega_S) \ge 3p_g(X) - 3.
	$$
	Combine the above inequalities with \eqref{eq: volume extension} and \eqref{eq: Severi q=1}. It follows that
	$$
	K_X^3 \ge \frac{1}{3}K_X^3 + \frac{1}{6} K_{S_0}^2 \ge \frac{2}{3} \left(\chi(\omega_X) + \chi(\omega_S) - 2\right) \ge 2p_g(X) - \frac{10}{3}.
	$$
	
	To prove the above claim, suppose first that $C_0^2 = 0$. Then $\psi|_S$ descends to a fibration $h: S_0 \to H'$ with $C_0$ as a general fibre. By the adjunction and parity, $g(C_0) = 2$ or $3$. Since $q(S_0) > 0$, by \cite[Corollary 1]{Xiao}, we have
	$$
	K_{S_0}^2 + 8 \left(g(C_0) - 1\right) = K_{S_0/H'}^2 \ge 4 \deg h_* \omega_{S_0/H'} = 4 \chi(\omega_S) + 4 \left(g(C_0) - 1\right).
	$$
	That is,
	$$
	K_{S_0}^2 \ge 4 \chi(\omega_S) - 4\left(g(C_0) - 1\right) \ge 4\chi(\omega_S) - 8.
	$$
	Suppose that $C_0^2 > 0$. Then by parity, either $(K_{S_0} \cdot C_0) = 3$ and $C_0^2 = 1$, or $(K_{S_0} \cdot C_0) = 4$ and $C_0^2 = 2$. By the Hodge index theorem, we always have $K_{S_0}^2 \le 9$. By \cite[Lemma 14]{Bombieri}, we know that $\chi(\omega_S) \le 4$ and
	$$
	K_{S_0}^2 \ge 2\chi(\omega_S) \ge 4\chi(\omega_S) - 8.
	$$
	Thus the claim holds, and the proof of the lemma is completed.
\end{proof}

\begin{prop} \label{prop: q=1 canonical dim 2}
	Suppose that $q(X) = 1$ and $\dim \Sigma = 2$. Then
	$$
	K_X^3 \ge \min\left\{\frac{5}{2}p_g(X) - 5,  2p_g(X) - \frac{10}{3}\right\}.
	$$
	In particular, if $p_g(X) \ge 6$, then $K_X^3 > \frac{4}{3} p_g(X)$.
\end{prop}

\begin{proof}
	This is just a combination of Lemma \ref{lem: q=1 canonical dim 2 case 1} and \ref{lem: q=1 canonical dim 2 case 2}.
\end{proof}

\subsection{The case when $\dim \Sigma = 3$} We now treat the case when $\dim \Sigma = 3$.

\begin{prop} \label{prop: q=1 canonical dim 3}
	Suppose that $\dim \Sigma = 3$. Then $K_X^3 \ge 2p_g(X) - 3$.
\end{prop}

\begin{proof}
	Let $S \in |M|$ be a general member. Then $S$ is a smooth surface of general type. Since $S$ admits a surjective morphism $f'|_S: S \to B$, we know that $q(S) > 0$. Let $h: S \to B'$ be the Stein factorization of $f'|_S$ with a general fibre $C$. Then $g(C) \ge 2$. Let $\sigma: S \to S_0$ be the contraction onto the minimal model of $S$. Since $g(B') > 0$, $h$ naturally descends to a fibration $h_0: S_0 \to B'$ with a general fibre $C_0 = \sigma_*C$.
	
	Since $\phi_M|_C$ is a finite morphism, we know that $h^0(C, M|_C) \ge 2$. In particular, $(M \cdot C) \ge 2$. By the adjunction, we deduce that
	$$
	\deg K_C = (K_S \cdot C) = \left( (K_{X'} + M) \cdot C \right) \ge 2(M \cdot C) \ge 4.
	$$
	Thus $g(C) \ge 3$. By \cite[Theorem 2]{Xiao}, we know that
	$$
	K_{S_0}^2 \ge \left(\frac{4g(C_0) - 4}{g(C_0)}\right) \chi(\omega_S) \ge \frac{8}{3} \chi(\omega_S).
	$$
	As in the proof of Proposition \ref{prop: albdim 2 canonical dim 3}, we still have
	$$
	\chi(\omega_X) + \chi(\omega_S) \ge 4p_g(X) - 6.
	$$
	Combine the above inequalities with \eqref{eq: volume extension} and \eqref{eq: Severi q=1}. It follows that
	$$
	K_X^3 \ge \frac{1}{4}K_X^3 + \frac{3}{16} K_{S_0}^2 \ge \frac{1}{2}(\chi(\omega_X) + \chi(\omega_S)) \ge 2p_g(X) - 3.
	$$
	Thus the proof is completed.
\end{proof}

\subsection{The case when $X$ is Gorenstein} Finally, we treat the case when $X$ is Gorenstein.

\begin{lemma} \label{lem: Sym2 injective}
	Suppose that $X$ is Gorenstein. If the map $\Sym^2 H^0(F, K_F) \to H^0(F, 2K_F)$ is injective, then
	$$
	K_X^3 \ge \left(2p_g(F) - 4\right) \deg f_*\omega_X + h^1(B, f_* \omega_X) \left(h^1(B, f_* \omega_X) + 1\right).
	$$
\end{lemma}

\begin{proof}
	Denote $q_f = h^1(B, f_* \omega_X)$. By \cite[Theorem 3.1]{Fujita}, we have
	$$
	f_* \omega_X = \CO_B^{\oplus q_f} \oplus \CV,
	$$
	where $\rank \CV = p_g(F) - q_f$ and $h^1(B, \CV) = 0$. By the assumption, we have $\Sym^2f_*\omega_X\subseteq f_*\omega_X^{\otimes 2}$. Denote by $\CF$ the saturation of $\Sym^2 f_*\omega_X$ in $f_*\omega_X^{\otimes 2}$. Then
	we have the following short exact sequence
	$$
	0 \to \CF \to f_* \omega_X^{\otimes 2} \to \CQ \to 0.
	$$
	By the Kawamata-Viehweg vanishing theorem and the Leray spectral sequence, we know that $h^1(B, f_* \omega_X^{\otimes 2}) = 0$. Via the long exact sequence, we deduce that $h^1(B, \CQ) = 0$ and 
	$h^0(B, \CQ)  \ge h^1(B, \CF)$. Thus by the Riemann-Roch theorem, we have
	$$
	\deg f_* \omega_X^{\otimes 2}  =  \deg \CF + \deg \CQ = h^0(B, \CF) - h^1(B, \CF) + h^0(B, \CQ)
	\ge h^0(B, \CF). 
	$$
    Note that 
    $$
    \deg \left(\Sym^2 f_* \omega_X\right) = \left(p_g(F) + 1\right) \deg f_* \omega_X
    $$ 
    and 
    $$
    h^1 \left(B, \Sym^2 f_* \omega_X\right) \ge h^1 \left(B, \Sym^2 \CO_B^{\oplus q_f}\right) = \frac{q_f(q_f + 1)}{2}.
    $$ 
    By the Riemann-Roch theorem, we know that
    $$
    h^0(B, \CF) \ge h^0 \left(B, \Sym^2 f_*\omega_X\right) = \left(p_g(F) + 1\right) \deg f_* \omega_X + \frac{q_f(q_f + 1)}{2}.
    $$
    On the other hand, since $h^1(B, f_* \omega_X^{\otimes 2}) = 0$, by the Riemann-Roch formula, we know that
    $$
    \deg f_* \omega_X^{\otimes 2} = h^0(B, f_* \omega_X^{\otimes 2}) = h^0(X, 2K_X) = \frac{1}{2} K_X^3 + 3 \chi(\omega_X).
    $$
    Moreover, since $q(X) = 1$, by the Leray spectral sequence, we have
    $$
    \chi(\omega_X) = p_g(X) - h^2(X, \CO_X) \le h^0(B, f_*\omega_X) - h^1(B, f_* \omega_X) = \deg f_* \omega_X.
    $$
    Combine the above (in)equalities together, and it follows that
    $$
    K_X^3 \ge 2 \deg f_* \omega_X^{\otimes 2} - 6 \chi(\omega_X)
    \ge \left(2p_g(F) - 4\right) \deg f_*\omega_X + q_f(q_f + 1).
    $$
    The proof is completed.
\end{proof}

\begin{lemma} \label{lem: q=1 pg=2}
	Suppose that $X$ is Gorenstein with $p_g(X) = 2$. Then $K_X^3 \ge 4$.
\end{lemma}

\begin{proof}
	Since $K_X^3$ is an even integer, to prove this lemma, it suffices to prove that $K_X^3 > 2$. Now $\chi(\omega_X) = p_g(X) - h^2(X, \CO_X) > 0$. By the Leray spectral sequence, we know that $h^1(B, f_*\omega_X) \le h^2(X, \CO_X) \le 1$.
	
	If $h^1(B, f_*\omega_X) = 0$, by \eqref{eq: slope q=1} and the Riemann-Roch theorem, we have 
	$$
	K_X^3 \ge 2\deg f_* \omega_X = 2h^0(B, f_*\omega_X) = 2p_g(X).
	$$
	From now on, we assume that $h^1(B, f_*\omega_X) = 1$. By the Riemann-Roch theorem, $\deg f_* \omega_X = h^0(B, f_*\omega_X) - h^1(B, f_*\omega_X) = 1$. Thus $f_*\omega_X$ cannot be a line bundle. In particular, $p_g(F) = \rank f_* \omega_X \ge 2$. If $p_g(F) = 2$, by \cite[Theorem 2.1]{Fujita}, $f_* \omega_X = \CO_B \oplus \CL$ where $\CL$ is a line bundle on $B$ with $\deg \CL = 1$. By \cite[Corollary 3.5]{Miyaoka}, $K_X - F$ is pseudo-effective. By \cite[Theorem 1.4]{Ohno} and \cite[Lemma 4.1]{Chen_Hacon2}, $K_X - \epsilon F$ is nef for some $\epsilon > 0$. Thus
	$$
	K_X^3 > \left((K_X - \epsilon F)^2 \cdot K_X\right) \ge \left((K_X - \epsilon F)^2 \cdot F\right) = K_F^2 \ge 2,
	$$
	where the last inequality holds since $F$ is not a $(1, 2)$-surface. If $p_g(F) = 3$, then the natural map $\Sym^2 H^0(F, K_F) \to H^0(F, 2K_F)$ is injective. Thus by Lemma \ref{lem: Sym2 injective}, we have
	$$
	K_X^3 \ge 2 \deg f_* \omega_X + 2 = 4.
	$$
	If $p_g(F) \ge 4$, since $f_* \omega_{X}^{\otimes m}$ is ample for any $m \ge 2$ \cite[Lemma 4.1]{Chen_Hacon2}, by \cite[Main Theorem 1]{Ohno} and the argument in the proof of \cite[Proposition 2.1]{Ohno} (see the inequality $(*)$ there), we know that
	$$
	K_X^3 > \left(\frac{4p_g(F) - 8}{p_g(F)}\right) \deg f_* \omega_X \ge 2 \deg f_* \omega_X = 2.
	$$
	As a result, we always have $K_X^3 > 2$, and the proof is completed.
\end{proof}

\begin{prop} \label{prop: q=1 Gorenstein}
	Suppose that $X$ is Gorenstein. Then
	$$
	K_X^3 \ge \left\{
	\begin{array}{ll}
		4, & p_g(X) = 2; \\
		2p_g(X) - 2, & p_g(X) \ge 3.
	\end{array}
	\right.
	$$
\end{prop}

\begin{proof}
	As $X$ is Gorenstein, we know that $K_X^3 > 0$ is even \cite[\S 2.2]{Chen_Chen_Zhang}. Denote by $\Sigma$ the canonical image of $X$. If $\dim \Sigma = 3$, then the result follows from Proposition \ref{prop: q=1 canonical dim 3}. If $\dim \Sigma = 2$, then the result follows from Proposition \ref{prop: q=1 canonical dim 2}. If $\dim \Sigma = 1$, then the result follows from Proposition \ref{prop: q=1 canonical dim 1} for $p_g(X) \ge 3$ and Lemma \ref{lem: q=1 pg=2} for $p_g(X) = 2$. Thus the whole proof is completed.
\end{proof}

\subsection{Proof of Theorem \ref{thm: q=1 2}} Now it is clear that Theorem \ref{thm: q=1 2} is a combination of Proposition \ref{prop: q=1 canonical dim 1}, \ref{prop: q=1 canonical dim 2}, \ref{prop: q=1 canonical dim 3} and \ref{prop: q=1 Gorenstein}.

\section{Proof of main theorems}
We are now ready to prove the three main theorems in this paper.

\subsection{Proof of Theorem \ref{thm: main1}}
Since the canonical volume and the geometric genus are birational invariants, we may assume that $X$ is minimal.

If $q(X) \ge 2$, the result follows from Theorem \ref{thm: q>1}. If $q(X) = 1$, the result follows from Theorem \ref{thm: q=1 1} and \ref{thm: q=1 2}.

\subsection{Proof of Theorem \ref{thm: main2}} As in the proof of Theorem \ref{thm: main1}, we may assume that $X$ is minimal with $K_X^3 = \frac{4}{3}p_g(X)$ and $p_g(X) \ge 16$. By Theorem \ref{thm: q>1}, $q(X) = 1$. Thus the general Albanese fibre $F$ is a minimal surface of general type. By Theorem \ref{thm: q=1 2}, $F$ is a $(1, 2)$-surface. Then we apply Theorem \ref{thm: q=1 1} to deduce that $h^2(X, \CO_X) = 0$. As a result, we have $K_X^3 = \frac{4}{3}\chi(\omega_X)$.
By \cite[Theorem 1.3]{Hu_Zhang2}, the canonical model of $X$ has an explicit description. 

\subsection{Proof of Theorem \ref{thm: main3}}
We may again assume that $X$ is minimal. Then the result just follows from Theorem \ref{thm: q>1}.

\bibliography{Ref_Noether}
\bibliographystyle{amsplain}

\end{document}